\documentclass[a4paper,11pt]{article} 

\usepackage{hyperref}

\usepackage[a4paper,margin=2.7cm]{geometry}
\usepackage{titlesec}
\titleformat{\section}{\LARGE\center\bfseries\scshape}{\thesection.}{.7em}{}
\titlespacing*{\section}{0pt}{3.5ex plus 0ex minus 0ex}{1.5ex plus 0ex}
\titleformat{\subsection}{\Large\center\bfseries\scshape}{\thesubsection.}{.7em}{}
\titlespacing*{\subsection}{0pt}{3.5ex plus 0ex minus 0ex}{1.5ex plus 0ex}

\makeatletter
\g@addto@macro\normalsize{%
  \setlength\abovedisplayskip{10pt}
  \setlength\belowdisplayskip{10pt}	
  \setlength\abovedisplayshortskip{10pt}
  \setlength\belowdisplayshortskip{10pt}}
\makeatother

\usepackage{tocloft}

\usepackage[english]{babel}
\usepackage[T1]{fontenc}
\usepackage{amsfonts}
\usepackage{mathrsfs}
\usepackage{bbm}
\usepackage{amssymb}
\usepackage{amsmath}

\usepackage{changebar}
\usepackage{enumitem}
\setlist{nolistsep}
\usepackage{amsthm}
\usepackage[capitalize]{cleveref}
\usepackage{chngcntr}
\usepackage[normalem]{ulem}

\newtheoremstyle{plain}{1mm}{2mm}{\slshape}{}{\color{Blue}\bfseries}{.}{.5em}{}
\newtheoremstyle{definition}{1mm}{2mm}{}{}{\color{Blue}\bfseries}{.}{.5em}{}
\theoremstyle{plain}
\newtheorem*{Thm}{Theorem}	
\newtheorem{Theorem}{Theorem}[section]

\newtheorem{Lemma}[Theorem]{Lemma}
\newtheorem{Proposition}[Theorem]{Proposition}
\newtheorem{Corollary}[Theorem]{Corollary}

\theoremstyle{definition}
\newtheorem{Definition}[Theorem]{Definition}

\newtheorem{Example}[Theorem]{Example}
\newtheorem{Question}[Theorem]{Question}

\theoremstyle{plain} 
\newcounter{MainTheoremCounter}

\newtheorem{Maintheorem}[MainTheoremCounter]{Theorem}
\newtheorem{Maincorollary}[MainTheoremCounter]{Corollary}

\theoremstyle{plain}
\newtheorem*{namedthm}{\namedthmname}
\newcounter{namedthm}
	

\usepackage{xcolor}

\definecolor{Scarlet}{rgb}{0.65, 0.10, 0.0}
\definecolor{Blue}{rgb}{0.0, 0.05, 0.39}
\definecolor{Green}{rgb}{0.3, 0.6 ,0.2}

\hypersetup{	citecolor = Scarlet,colorlinks,
			linkcolor = black,
			urlcolor = Scarlet}

\newcommand{\N}{\mathbb{N}}
\newcommand{\Z}{\mathbb{Z}}
\newcommand{\R}{\mathbb{R}}
\newcommand{\C}{\mathbb{C}}
\newcommand{\Q}{\mathbb{Q}}
\newcommand{\Hilb}{\mathscr{H}}

\newcommand{\define}[1]{\emph{#1}}

\newcommand{\la}{\langle}
\newcommand{\ra}{\rangle}
\renewcommand{\epsilon}{\varepsilon}
\renewcommand{\leq}{\leqslant}
\renewcommand{\geq}{\geqslant}
\renewcommand{\setminus}{{\backslash}}
\renewcommand{\Re}{{\rm Re}}

\renewcommand{\colon}{\nobreak\mskip2mu\mathpunct{}\nonscript\mkern-\thinmuskip{:}\mskip6muplus1mu\relax}

\newcommand{\xbm}{(X,\mathcal{B},\mu)}

\newcommand{\xbmt}{(X,\mathcal{B},\mu,T)}

\newcommand{\mob}{\boldsymbol{\mu}}
\newcommand{\lio}{\boldsymbol{\lambda}}
\newcommand{\tot}{\boldsymbol{\varphi}}

\newcommand{\vep}{\varepsilon}
\newcommand{\sB}{\mathscr{B}}

\newcommand{\cf}{\mathcal{F}}
\newcommand{\1}{\mathbbm{1}}
\newcommand{\cm}{\mathcal{M}}
\newcommand{\cn}{\mathcal{N}}

\newcommand{\ca}{\mathcal{A}}

\renewcommand{\P}{\mathbb{P}}
\newcommand{\T}{\mathbb{T}}
\newcommand{\rfctn}{H}
\newcommand{\s}{F}
\renewcommand{\t}{G}
\newcommand{\G}{\mathbf{G}}

\newcommand{\Oh}{{\rm O}}
\newcommand{\oh}{{\rm o}}

\newcommand{\Hardy}{\mathcal{H}}

\newcommand{\Dzero}{\mathcal{D}^{(1)}}
\newcommand{\Done}[1]{\mathcal{D}^{(#1)}}
\newcommand{\DtwoA}{\mathcal{E}_{\text{\normalfont c.pt.}}}
\newcommand{\DtwoB}{\mathcal{E}_{\text{\normalfont pol}}}
\newcommand{\Dthree}[1]{\mathcal{E}_{\text{\normalfont c.pt.}}^{(#1)}}
\newcommand{\Dfour}{\mathcal{E}_{\partial}}

\newcommand{\DtwoC}{\mathcal{E}_{\text{\normalfont Jor}}}

\begin{document}
\allowdisplaybreaks

\title{A generalization of K{\'a}tai's orthogonality criterion with applications
}
\author{V.\ Bergelson\thanks{The first author gratefully acknowledges the support of the NSF under grant DMS-1500575.} \and J.\ Ku\l aga-Przymus\thanks{Research supported by Narodowe Centrum Nauki UMO-2014/15/B/ST1/03736 and the European Research Council (ERC) under the European Union's Horizon 2020 research and innovation programme (grant agreement No 647133 (ICHAOS)).} \and M.\ Lema\'nczyk\thanks{Research supported by Narodowe Centrum Nauki UMO-2014/15/B/ST1/03736 and the EU grant ``AOS'', FP7-PEOPLE-2012-IRSES, No 318910.} \and F.\ K.\ Richter}

\date{\small \today}

\maketitle
\begin{abstract}
\noindent

We study properties of arithmetic sets coming from multiplicative number theory and obtain applications in the theory of uniform distribution and ergodic theory.
Our main theorem is a generalization of K{\'a}tai's orthogonality criterion. Here is a special case of this theorem:
\begin{Thm}
Let $a\colon\N\to\C$ be a bounded sequence satisfying
\begin{equation*}
\sum_{n\leq x} a(pn)\overline{a(qn)} = \oh(x),~\text{for all distinct primes $p$ and $q$.}
\end{equation*}
Then for any multiplicative function $f$ and any $z\in\C$ the indicator function of the level set $E=\{n\in\N:f(n)=z\}$ satisfies
\begin{equation*}
\sum_{n\leq x} \1_E(n)a(n)=\oh(x).
\end{equation*}
\end{Thm}

With the help of this theorem one can show that if $E=\{n_1<n_2<\ldots\}$ is a level set of a multiplicative function having positive upper density, then for a large class of sufficiently smooth functions $h:(0,\infty)\to\R$ the sequence $(h(n_j))_{j\in\N}$ is uniformly distributed $\bmod~1$. This class of functions $h(t)$ includes: all polynomials $p(t)=a_kt^k+\ldots+a_1t+a_0$ such that at least one of the coefficients $a_1,a_2,\ldots,a_k$ is irrational, $t^c$ for any $c>0$ with $c\notin \N$, $\log^r(t)$ for any $r>2$, $\log(\Gamma(t))$, $t\log(t)$, and $\frac{t}{\log t}$. The uniform distribution results, in turn, allow us to obtain new examples of ergodic sequences, i.e. sequences along which the ergodic theorem holds.
\end{abstract}
\small
\tableofcontents
\thispagestyle{empty}
\normalsize

\section{Introduction}
\label{sec:intro}

An arithmetic function $f\colon\N=\{1,2,\ldots,\}\to\C$ is called \define{multiplicative} if $f(1)=1$ and $f(m n)=f(m)\cdot f(n)$ for all relatively prime $m,n\in\N$ (and is called {\em completely multiplicative} if $f(m n)=f(m)\cdot f(n)$ for all $m,n\in\N$).
We start the discussion by formulating the following classical result of Daboussi.

\begin{Theorem}[cf.\ {\cite[Theorem 1]{DD82}}]
\label{thm:DD82-thm.1.}
Let $f\colon\N\rightarrow\C$ be a multiplicative function with $|f(n)|\leq 1$ for all $n\in\N$. Then for all irrational $\theta$,
$$
\sum_{n\leq x} f(n)e(\theta n)=\oh(x),
$$
where $e(x):=e^{2\pi i x}$ for all $x\in\R$.
\end{Theorem}

A nice (and shorter) proof of \cref{thm:DD82-thm.1.}, which also yields more general results (for instance $e(\theta n)$ replaced with $e(\theta n^2)$), was later discovered by K{\'a}tai \cite{Katai86}. The following theorem is the main technical result that K{\'a}tai uses to improve Daboussi's result and, in addition, to derive new results in the theory of equidistribution (in particular, it is proved in \cite{Katai86} that for any \define{additive function}\footnote{An arithmetic function $a\colon\N\to\R$ is called \define{additive} if $a(nm)=a(n)+a(m)$ for all $m,n$ with $\gcd(n,m)=1$.\label{ftnt-1}} $a\colon\N\to\R$ and any polynomial $p(t)=a_kt^k+\ldots+a_1t+a_0$ such that at least one of the coefficients $a_1,a_2,\ldots,a_k$ is irrational the sequence $a(n)+p(n)$ is \define{uniformly distributed mod $1$}\footnote{A real-valued sequence $(x_n)_{n\in\N}$ is called \define{uniformly distributed mod $1$} if for all continuous functions $f\colon[0,1)\to\C$ one has
$$\lim_{N\rightarrow\infty}\frac{1}{N}\sum_{n=1}^N f(\{x_n\})=\int_0^1 f(x)\, dx,$$
where for $y\in\R$ the expression $\{y\}$ denotes the fractional part of $y$.\label{ftnt-2}}.).

\begin{Theorem}[K{\'a}tai's orthogonality criterion, see {\cite{Katai86,BSZ13}}]\label{thm:org-Katai}
Let $a\colon\N\to\C$ be a bounded sequence satisfying
\begin{equation}
\label{KC}
\sum_{n\leq x} a(pn)\overline{a(qn)} = \oh(x),~\text{for all distinct primes $p$ and $q$.}
\end{equation}
Then for every multiplicative function $f\colon\N\rightarrow\C$ that is bounded in modulus by $1$, one has
\begin{equation}
\label{eq:kc-conclusion}
\sum_{n\leq x} f(n)a(n)=\oh(x).
\end{equation}
\end{Theorem}

Given a multiplicative function $f\colon\N\to\C$ and a point $z\in\C$ let $E(f,z)$ denote the set of solutions to the equation $f(n)=z$, i.e.,
\begin{equation*}
E(f,z):=\{n\in\N: f(n)=z\}.
\end{equation*}
We will refer to $E(f,z)$ as a \define{level set} of $f$.
While $E(f,z)$ is defined by means of the multiplicative structure of $\N$, it possesses many interesting properties from the viewpoint of additive integer arithmetic.

Our main result is a generalization of K{\'a}tai's orthogonality criterion in which the multiplicative function $f$ is replaced by the indicator function of a level set of $f$.
Actually, our result holds for sets that are more general than sets of the form $E(f,z)$.


\begin{Definition}
\label{def:c1c2}
\
\begin{description}
\item[{\big(Definition of $\Done{r}$\big).}]
For $r\in\N$ let $\Done{r}$ denote the collection of all sets of the from
$$E(f_1,\ldots,f_r,z_1,\ldots,z_r):=\{n\in\N:f_1(n)=z_1,\ldots,~f_r(n)=z_r\},$$
where $f_1,\ldots,f_r$ are arbitrary multiplicative functions and $z_1,\ldots,z_r$ are arbitrary complex numbers. It is clear that $\Done{1}\subset\Done{2}\subset\ldots$; we set $\Done{\infty}:=\bigcup_{r=1}^{\infty}\Done{r}$.
\item[{\big(Definition of $\DtwoA$\big).}]
A point $z\in\C$ is called a \define{concentration point} for $f\colon\N\to\C$ if $\sum_{p~\text{prime}\atop f(p)=z}\frac{1}{p}=\infty$ (cf.\ {\cite[Definition 3.9]{Ruzsa77}}).
We define $\DtwoA$ to be the collection of all sets of the from
$E(f,K):=\{n\in\N:f(n)\in K\}$, where $K$ is an arbitrary subset of $\C$ and $f\colon\N\to\C$ is a multiplicative function possessing at least one concentration point.
\item[{\big(Definition of $\DtwoB$\big).}]
A set $K\subset \C$ is an \emph{elementary set in polar coordinates} if it can be expressed as a finite union of sets of the form $\{re^{2\pi i\varphi}: r\in I_1,\varphi\in I_2 \}$, where $I_1$ and $I_2$ are (open, closed or half-open) intervals in $\R$.
Let $\DtwoB$ denote the collection of all sets of the form
$E(f,K):=\{n\in\N:f(n)\in K\}$, where $K$ is an elementary set in polar coordinates and $f$ is a multiplicative function bounded in modulus by $1$ and satisfying $\lim_{N\to\infty}\frac{1}{N}\sum_{n=1}^N|f(n)|\neq 0$ (note that this limit always exists by Wirsing's mean value theorem, see \cref{thm:wirsing} below).
\end{description}
\end{Definition}


The classes $\Done{\infty}$, $\DtwoA$ and $\DtwoB$ contain numerous classical sets originating in multiplicative number theory.
The following (admittedly long) list is comprised of representative examples of sets from these classes which will frequently appear in the next sections of the paper. A more detailed explanation why the sets in \ref{item:eg:multiplicative-fiber-0} - \ref{item:eg:multiplicative-fiber-7} below are indeed elements of $\Done{\infty}$, $\DtwoA$ or $\DtwoB$ is provided at the end of Subsection \ref{sec:a(f)} (see \cref{example:mf-d3d4}).

\begin{Example}\label{example:mf}
\
\begin{enumerate}
[label=\text{Ex.\ref{example:mf}.\arabic*:}, ref=\text{Ex.\ref{example:mf}.\arabic*}, leftmargin=*]
\item
\label{item:eg:multiplicative-fiber-0}
The set $Q$ of \define{squarefree numbers} belongs to $\Dzero$.
\item
\label{item:eg:multiplicative-fiber-2}
Let $\Omega(n)$ denote the number of prime factors of $n$ (counted with multiplicities) and $\omega(n)$ denote the number of distinct prime divisors of $n$ (without multiplicities). For any $b_1,b_2,r_1,r_2\in\N$, the sets
\begin{eqnarray*}
S_{\Omega,b_1,r_1}&:=&\{n\in\N: \Omega(n)\equiv r_1\bmod b_1\}
\\
S_{\omega,b_2,r_2}&:=&\{n\in\N: \omega(n)\equiv r_2\bmod b_2\}
\end{eqnarray*}
belong to $\Dzero$ and the sets
$$
S_{\omega,b_1,r_1}\cap S_{\Omega,b_2,r_2}=\{n\in\N: \omega(n)\equiv r_1\bmod b_1,~\Omega(n)\equiv r_2\bmod b_2\}
$$
belong to $\Done{2}$.
\item
\label{item:eg:multiplicative-fiber-4}
For any irrational $\alpha>0$ and any set $J\subset [0,1)$, the sets
\begin{eqnarray*}
S_{\Omega,\alpha,J}&:=&\{n\in\N: \Omega(n)\alpha\bmod 1 \in J\}
\\
S_{\omega,\alpha,J}&:=&\{n\in\N: \omega(n)\alpha\bmod 1\in J\}
\end{eqnarray*}
belong to $\DtwoA$ (cf.\ \cite{FH15arXiv}).
\item\label{item:eg:multiplicative-fiber-1}
For any $x\in(0,1)$, the set $\Phi_x:=\{n\in\N: \tot(n)<xn\}$ belongs to $\DtwoB$, where $\tot(n)$ is \define{Euler's totient function} (cf.\ \cite{Schoenberg28}).
\item\label{item:eg:multiplicative-fiber-5}
The set of \define{abundant numbers} $\mathscr{A}:=\{n\in\N:\boldsymbol{\sigma}(n)>2n\}$ and the set of \define{deficient numbers} $\mathscr{D}:=\{n\in\N:\boldsymbol{\sigma}(n)<2n\}$ belong to $\DtwoB$; here $\boldsymbol{\sigma}(n):=\sum_{d\mid n} d$ denotes the \define{sum of divisors function} (cf.\ \cite{Davenport33}). 
\item\label{item:eg:multiplicative-fiber-6}
Let $\boldsymbol{\tau}(n):=\sum_{d\mid n} 1$ be the \define{number of divisors function}.
For $b,r\in\N$ with $\gcd(r,b)=1$, the set
\begin{eqnarray*}
S_{\boldsymbol{\tau},b,r}&:=&\{n\in\N: \boldsymbol{\tau}(n)\equiv r\bmod b\}
\end{eqnarray*}
belongs to $\Done{t}$, where $t$ equals the number of generators of the group $(\Z/b\Z)^*$. More generally, $\{n\in\N: f(n)\equiv r\bmod b\}\in\Done{t}$ for any multiplicative function $f\colon\N\to\N$ (cf. \ref{item:eg:multiplicative-fiber-6-d3d4} in Subsection \ref{sec:a(f)}).
\item\label{item:eg:multiplicative-fiber-7}
If $E$ belongs to either $\Done{\infty}$, $\DtwoA$ or $\DtwoB$, then for any \define{multiplicative set}\footnote{A set $M\subset\N$ is called multiplicative if $1\in M$ and for all $m,n\in\N$ with $\gcd(m,n)=1$ one has $m\cdot n\in M$ if and only if $m\in M$ and $n\in M$. Equivalently, a set $M$ is multiplicative if and only if its indicator function $\1_M$ is a multiplicative function.\label{footnote-p3}} $M$ the set $E\cap M$ again belongs to $\Done{\infty}$, $\DtwoA$ or $\DtwoB$ respectively. Clearly, any subsemigroup of $(\N,\cdot)$ containing $1$ is a multiplicative set. Other examples include the set of $k$-free numbers.
\end{enumerate}
\end{Example}

\begin{Maintheorem}[A generalization of K{\'a}tai's orthogonality criterion]
\label{thm:Katai-for-sets}
Let $a\colon\N\to\C$ be a bounded sequence satisfying
\begin{equation*}
\sum_{n\leq x} a(pn)\overline{a(qn)} = \oh(x),~\text{for all distinct primes $p$ and $q$.}
\end{equation*}
If $E\subset\N$ belongs to one of the classes $\Done{\infty}$, $\DtwoA$ or $\DtwoB$ then
\begin{equation}
\label{eq:kc-conclusion-sets}
\sum_{n\leq x} \1_{E}(n)a(n)=\oh(x).
\end{equation}
\end{Maintheorem}

Note that one can quickly derive \cref{thm:org-Katai} from \cref{thm:Katai-for-sets}.
Indeed, any multiplicative function $f\colon\N\to\C$ that is bounded in modulus by $1$ can be uniformly approximated by finite linear combinations of functions of the form $\1_{E(f,K)}$, where $K$ is an elementary set in polar coordinates (and hence $E(f,K)\in\DtwoB$).

In \cref{sec:KOC} we also state and prove a generalization of \cref{thm:Katai-for-sets} in which the restrictions on $f$ and $K$ in the definition of $\DtwoA$ and $\DtwoB$ are slightly relaxed (see \cref{thm:Katai-for-sets-2}).
However, the restrictions on $f$ and $K$ in $\DtwoA$ and $\DtwoB$ cannot be dropped entirely, as there are multiplicative functions $f$ and sets $K\subset \C$ such that \eqref{eq:kc-conclusion-sets} does not hold for $E=E(f,K)$\footnote{Indeed, if there are no restrictions on $f$ or $K$ then any set $B\subset \N$ can be written in the from $E(f,K)$. Let $(\xi_{n})_{n\in\N}$ be a rationally independent family of irrational numbers in $[0,1)$, let $(p_n)_{n\in\N}$ be an enumeration of the prime numbers and define $f(p_1^{c_1}\cdot\ldots\cdot p_k^{c_k})=e(c_1\xi_{1}+\ldots+c_k\xi_k)$. Clearly, $f(n)\neq f(m)$ for all $n\neq m$ and therefore, if we set $K:=\{f(n):n\in B\}$, we get $E(f,K)=B$.}.

From \cref{thm:Katai-for-sets}, by setting $a(n)=e(n\theta)$, we immediately obtain the following generalization of \cref{thm:DD82-thm.1.}.  

\begin{Maincorollary}\label{cor:Katai-for-sets}
Suppose $E\subset\N$ belongs to one of the classes $\Done{\infty}$, $\DtwoA$ or $\DtwoB$. Then for any irrational $\theta$ we have
$$
\sum_{n\leq x} \1_E(n)e(\theta n)=\oh(x).
$$
\end{Maincorollary}

From \cref{cor:Katai-for-sets} we obtain an application to ergodic theory. We need first the following definition.

\begin{Definition}
\label{def:tot-erg-sequence}
A sequence $(n_j)_{j\in\N}$ in $\N$ is called \define{totally ergodic} if for any totally ergodic\footnote{A measure preserving system $\xbmt$ is called \define{totally ergodic} if for every $m\in\N$ the map $T^m\colon X\to X$ is ergodic.} measure preserving system $\xbmt$ and any $f\in L^2$ we have
$$
\lim_{N\to\infty}\frac{1}{N}\sum_{j=1}^N T^{n_j}f = \int_X f\, d\mu,
$$
where $Tf(x):=f(Tx)$ and the convergence takes place in $L^2\xbm$.
\end{Definition}

Using the spectral theorem, it is straightforward to show that a sequence $(n_j)_{j\in\N}$ is totally ergodic if and only if $(n_j\alpha)$ is uniformly distributed mod $1$ for all irrational $\alpha$. Thus \cref{cor:Katai-for-sets} yields the following result.

\begin{Maincorollary}\label{cor:M-sets-are-tot.erg.sequences}
Let $E=\{n_1<n_2<\ldots\}$ be a set that belongs to one of the classes $\Done{\infty}$, $\DtwoA$ or $\DtwoB$ and suppose $d(E)$ exists\footnote{For any $E\in\Done{r}$ it was shown by Ruzsa that the \define{natural density} $d(E):=\lim_{N\to\infty}\frac{|E\cap\{1,\ldots,N\}|}{N}$ exists (cf.\ {\cite[Corollary 1.6 and the subsequent remark]{Ruzsa77}}). The density of sets $E=E(f,K)$ belonging to $\DtwoA$ or $\DtwoB$ may not exist, but it exists for a rather wide family of sets $E(f,K)$, where the multiplicative function $f$ and the set $K$ are sufficiently regular. In particular, all sets appearing in \cref{example:mf} have positive natural density.} and is positive. Then $(n_j)_{j\in\N}$ is a totally ergodic sequence.
\end{Maincorollary}

\cref{thm:Katai-for-sets} also leads to new uniform distribution results involving functions from Hardy fields.
Let $G$ denote the set of all \define{germs}\footnote{A \define{germ
at $\infty$} is an equivalence class of functions
under the equivalence relationship
$(f\sim g) \Leftrightarrow \big(\exists t_0>0
~\text{such that}~f(t)=g(t)~\text{for all}~t\in (t_0,\infty)\big)$.} at $\infty$ of real valued functions defined
on some half-line $(t_0,\infty)\subset\R$.
Note that $G$ forms a ring under
pointwise addition and multiplication, which we denote by
$(G,+,\cdot)$.
Any subfield of the ring $(G,+,\cdot)$ that is closed under
differentiation is called a
\define{Hardy field}. 
By abuse of language, we say that a function $h\colon (0,\infty)\to\R$
belongs to some Hardy field $\Hardy$ (and write $f\in\Hardy$)
if its germ at $\infty$ belongs to $\Hardy$.
See \cite{Boshernitzan81,Boshernitzan82,Boshernitzan94}
and some references therein for
more information on Hardy fields.

Here are some classical examples of functions from Hardy fields.
\begin{itemize}
\item
the class of \define{logarithmico-exponential functions} introduced by Hardy in \cite{Hardy12,Hardy10}, which consists of all functions that can be
obtained from polynomials with real coefficients, $\log(t)$ and $\exp(t)$
using the standard arithmetical operations $+$,$-$,$\cdot$,$/$ and
the operation of composition (e.g. $\frac{p(t)}{q(t)}$ for all $p,q\in\R[t]$, $t^c$ for all $c\in\R$, $\frac{\log t}{t}$, $t\log t$, etc.).
\item
the Gamma function $\Gamma(t)$, the Riemann zeta function $\zeta(t)$, and the logarithmic integral function $\text{Li}(t)$.
\end{itemize}

Given two functions $f,g\colon (0,\infty)\to\R$ we write $f(t)\prec g(t)$
if $\frac{g(t)}{f(t)}\to\infty$ as $t\to\infty$. We will say that a function $f(t)$ has \define{polynomial growth} if there exists $k\in\N$ such that $f(t)\prec t^k$.

The next theorem, which is proved in \cref{subsec:ud}, follows from \cref{thm:Katai-for-sets} using elementary computations and results of Boshernitzan \cite{Boshernitzan94}.

\begin{Maintheorem}
\label{thm:uniform-distribution}
Let $E=\{n_1<n_2<\ldots\}$ be a set that belongs to either $\Done{\infty}$, $\DtwoA$ or $\DtwoB$.
Suppose $h\colon (0,\infty)\to\R$ belongs to a Hardy field, has polynomial growth and satisfies $|h(t)- r(t)|\succ \log^2(t)$ for all polynomials $r\in\Q[t]$.
If $d(E)$ exists and is positive then the sequence $\big(h(n_j)\big)_{j\in\N}$ is uniformly distributed mod $1$. \end{Maintheorem}

In the following corollary we give a sample of particularly interesting cases to which \cref{thm:uniform-distribution} applies.

\begin{Maincorollary}
\label{cor:uniform-distribution}
Let $E=\{n_1<n_2<\ldots\}$ be a set that belongs to one of the classes $\Done{\infty}$, $\DtwoA$ or $\DtwoB$ and suppose $d(E)$ exists and is positive. Then
\begin{itemize}
\item
the sequence $\big(p(n_j)\big)_{j\in\N}$ is uniformly distributed mod $1$
for any polynomial $p(t)=a_kt^k+\ldots+a_1t+a_0$ such that at least one of the coefficients $a_1,a_2,\ldots,a_k$ is irrational;
\item
the sequence $(n_j^c )_{j\in\N}$ is uniformly distributed mod $1$ for any positive real number $c$ that is not an integer.
\item
the sequence $(\log^r n_j)_{j\in\N}$ is uniformly distributed mod $1$ for any $r>2$.
\end{itemize}
\end{Maincorollary}

\cref{thm:uniform-distribution} also yields applications to ergodic theory.

\begin{Definition}[cf.\ \cref{def:tot-erg-sequence} above]
\label{def:ergodic-sequence}
A sequence $(n_j)_{j\in\N}$ of integers is called an \define{ergodic sequence} if for any ergodic probability measure preserving system $\xbmt$ and any $f\in L^2$ we have
$$
\lim_{N\to\infty}\frac{1}{N}\sum_{j=1}^N T^{n_j}f = \int_X f\, d\mu,
$$
where convergence takes place in $L^2\xbm$.
\end{Definition}

Using the spectral theorem and standard techniques in ergodic theory one can derive from \cref{thm:uniform-distribution} the following corollary.

\begin{Maincorollary}
\label{thm:ergodic-sequence}
Let $E=\{n_1<n_2<\ldots\}$ be a set that belongs to one of the classes $\Done{\infty}$, $\DtwoA$ or $\DtwoB$. Suppose $h\colon (0,\infty)\to\R$ belongs to a Hardy field, has polynomial growth and satisfies either $\log^2 t\prec h(t)\prec t$ or $t^k\prec h(t)\prec t^{k+1}$ for some $k\in\N$.
If $d(E)$ exists and is positive then the sequence $\big(\lfloor h(n_j)\rfloor\big)_{j\in\N}$ is an ergodic sequence.
\end{Maincorollary}

\paragraph{Structure of the paper:}
\

In \cref{sec:prelim} we review basic results and facts regarding multiplicative and additive functions, which are needed in the subsequent sections.

In \cref{sec:Katai-all} we establish some generalizations of the K{\'a}tai orthogonality criterion and, in particular, give a proof of \cref{thm:Katai-for-sets}.

Sections \ref{subsec:ud} and \ref{sec:single-rec} contain numerous applications of our main results to the theory of uniform distribution and to ergodic theory. \cref{thm:uniform-distribution} is proved in \cref{subsec:ud} and \cref{thm:ergodic-sequence} is proved in \cref{sec:single-rec}.


\section{Preliminaries}\label{sec:prelim}

In this section we present a brief overview of classical results and facts from multiplicative number theory that will be used in subsequent sections.

\subsection{Multiplicative functions}
\label{subsec:multiplicative-functions}

Define
\[
\cm:=\Big\{f\colon\N\to\C:\text{$f$ is multiplicative and}~\sup_{n\in\N}|f(n)|\leq1\Big\}.
\]
The following sample amply demonstrates the diversity of multiplicative functions belonging to $\cm$; these functions will frequently appear in the later sections.

\begin{Example}
\label{Ex:multiplicative-functions}
\
\begin{enumerate}
[label=\text{Ex.\ref{Ex:multiplicative-functions}.\arabic*:}, ref=\text{Ex.\ref{Ex:multiplicative-functions}.\arabic*}, leftmargin=*]
\item\label{item:eg:multiplicative-functions-1}
The \define{Liouville function}
$\lio$ is defined as $\lio(n):=(-1)^{\Omega(n)}$ and is completely multiplicative (for the definition of $\Omega(n)$ see \cref{example:mf}).
\item\label{item:eg:multiplicative-functions-1.5}
The \define{M{\"o}bius function} $\mob$ is defined as $\mob(n):=\lio(n)$ if $n$ is squarefree and $\mob(n):=0$ otherwise. Note that $\mob$ is multiplicative but not completely multiplicative.
\item\label{item:eg:multiplicative-functions-2}
Let $\tot$ denote Euler's totient function. Clearly, $\frac{\tot(n)}{n}\in\cm$.
\item\label{item:eg:multiplicative-functions-3}
An arithmetic function $\chi$ is called a \define{Dirichlet character} if there exists a number $d\in\N$, called a \define{modulus of $\chi$}, such that
\begin{enumerate}
[label=\text{(\arabic*)\quad}, ref=\text{(\arabic*)}, leftmargin=*]
	\item 	$\chi(n+d)=\chi(n)$ for all $n\in\N$;
	\item 	$\chi(n)=0$ whenever $\gcd(d,n)>1$, and $\chi(n)$ is a $\tot(d)$-th root of unity whenever $\gcd(d,n)=1$;
	\item	$\chi(nm)=\chi(n)\chi(m)$ for all $n,m\in\N$.
\end{enumerate}
Any Dirichlet character is periodic and completely multiplicative. Also $\chi\colon \N\to\C$ is a Dirichlet character of modulus $k$ if and only if there exists a group character $\widetilde{\chi}$ of the multiplicative group $(\Z/k\Z)^*$ such that $\chi(n)=\widetilde{\chi}(n~{\rm mod}~k)$ for all $n\in\N$.
\item\label{item:eg:multiplicative-functions-4}
An \define{Archimedean character} is a function of the form $n\mapsto n^{it}=e^{it\log n}$ with $t\in\R$. Any Archimedean character
is completely multiplicative and takes values in the unit circle.
\item\label{item:eg:multiplicative-functions-5}
Throughout this paper we identify the torus $\T:=\R/\Z$ with the unit interval $[0,1)\bmod 1$ or, when convenient, with the unit circle in the complex plane. Given $\xi\in\T$, let us define the multiplicative functions $\boldsymbol{\kappa}_\xi$, $\lio_{\xi}$ and $\mob_{\xi}$ as
$$
\boldsymbol{\kappa}_\xi(n):=e( \xi\omega(n)),\qquad
\lio_{\xi}(n):=e( \xi\Omega(n))
$$
and
$$
\mob_{\xi}(n):=
\begin{cases}
e(\xi\Omega(n)),&\text{if $n$ is squarefree}
\\
0,&\text{otherwise.}
\end{cases}
$$
It is clear that $\boldsymbol{\kappa}_\xi,\lio_{\xi},\mob_{\xi}\in\cm$.
\end{enumerate}
\end{Example}

For $f\in\cm$ let $M(f)$ denote the
\define{mean value} of $f$ whenever it exists, i.e.,
\begin{equation}
\label{eqn:mean-value}
M(f):=\lim_{N\to\infty}\frac1N\sum_{n=1}^N f(n).
\end{equation}
Note that the mean of a multiplicative function does not always exist (take, for example, Archimedean characters, cf.\ \cite[Section 4.3]{GSdraft}).

In the 1960s the study of mean values of multiplicative functions was catalyzed by the works of D{'e}lange, Wirsing and Hal{\'a}sz \cite{Delange61,Halasz68,Wirsing61}.
For real-valued functions in $\cm$ Wirsing showed that the mean value always exists:

\begin{Theorem}[Wirsing; see \cite{Wirsing61} and {\cite[Theorem 6.4]{Elliott79}}]
\label{thm:wirsing}
For any real-valued $g\in\cm$ the mean value $M(g)$ exists.
\end{Theorem}

The next theorem is due to Hal{\'a}sz \cite{Halasz68} and provides easy to check (necessary and sufficient) conditions for $M(g)$ to exist. We use $\P$ to denote the set of prime numbers.

\begin{Theorem}[Hal{\'a}sz; see {\cite[Theorem 6.3]{Elliott79}}]
\label{thm:FH-thm.2.9}
Let $g\in\cm$. Then the mean value $M(g)$
exists if and only if one of the following mutually exclusive conditions is satisfied:
\begin{enumerate}
[label=\text{(\roman{enumi})}, ref=\text{(\roman{enumi})}, leftmargin=*]
\item\label{item:E79-thm6.3-i}
there is at least one positive integer $k$ so that $g(2^k)\neq-1$ and, additionally, the series $\sum_{p\in\P}\frac1p(1-g(p))$ converges;
\item
\label{item:E79-thm6.3-iii}
there is a real number $t$ such that $\sum_{p\in\P}\frac1p(1-\Re(g(p)p^{it}))$ converges and, moreover, for each positive integer $k$ we have $g(2^k)=-2^{itk}$;
\item
\label{item:E79-thm6.3-iv}
$\sum_{p\in\P}\frac1p(1-\Re(g(p)p^{it}))=\infty$ for each $t\in\R$.
\end{enumerate}
When condition \ref{item:E79-thm6.3-i} is satisfied then $M(g)$ is non-zero and can be computed explicitly using the formula
\begin{equation}
\label{eq:mean-value-0}
M(g)=\prod_{p\in\P}\left(1-\frac1p\right)\left(1+\sum_{m=1}^\infty p^{-m}g(p^m)\right).
\end{equation}
In the case when $g$ satisfies either \ref{item:E79-thm6.3-iii} or \ref{item:E79-thm6.3-iv} then the mean value $M(g)$ equals zero.
\end{Theorem}

Throughout the paper, given a bounded arithmetic function $f\colon\N\to\C$ we use $\|f\|_1$ to denote the seminorm
$$
\|f\|_1:=\limsup_{N\to\infty}\frac{1}{N}\sum_{n=1}^N |f(n)|.
$$ 

\begin{Corollary}[see {\cite[Lemma 2.9]{BKLR17arXiv}}]\label{lem:lH2-C}
Suppose $f\in\cm$. Then $\|f\|_1=0$ if and only if
$\sum_{p\in\P}\frac1p\big(1-|f(p)|\big)=\infty$.
\end{Corollary}

\begin{Example}
Consider the multiplicative function $\tfrac{\tot(n)}{n}$ of \ref{item:eg:multiplicative-functions-2} on page \pageref{item:eg:multiplicative-functions-2}. By \cref{thm:wirsing} we have that $M\left(\frac{\tot(n)}{n}\right)$ exists.
\cref{lem:lH2-C} implies that $M\left(\frac{\tot(n)}{n}\right)$ is non-zero. Indeed, $\sum_{p\in\P}\frac{1}{p}\left(1-\frac{\tot(p)}{p}\right)=\sum_{p\in\P}\frac{1}{p^2}<\infty$ and therefore, by \cref{lem:lH2-C}, $\left\|\frac{\tot(n)}{n}\right\|_1>0$. Hence the mean value of $\frac{\tot(n)}{n}$ is positive.
\end{Example}

\subsection{Additive functions with values in $\T$}\label{sec:distribution-additive-functions}

An arithmetic function $a\colon \N\to\T$ is called \define{additive} if $a(n\cdot m)=a(n)+a(m)\bmod 1$ for all $m,n$ with $\gcd(n,m)=1$. Note that for every additive function $a\colon \N\to\T$ the function $f\colon \N\to\{z\in\C:|z|=1\}\subset\C$ defined as
$$
f(n):=e(a(n))=e^{2\pi i a(n)}
$$
is a multiplicative function.

\begin{Definition}
\label{def:nu-u.d.}
Let $\nu$ be a Borel probability measure on $\T$ and let $x\colon \N\to\T$. The sequence $x$ has \define{limiting distribution $\nu$} if for all continuous functions $F\in C(\T)$,
$$
\lim_{N\to\infty}\frac{1}{N}\sum_{n=1}^N F(x(n))=\int_\T F\, d\nu.
$$
If $\nu$ is the Lebesgue measure on $\T$, then $x(n)$ is said to be \define{uniformly distributed in $\T$}.
\end{Definition}

\begin{Theorem}[see {\cite[Theorem 8.1, Theorem 8.2 and Remark after Theorem 8.2]{Elliott79}}]\label{thm:Elliott-thm8.1-and-thm8.2}
Let $a\colon \N\to \T$ be an additive function and $f(n):=e(a(n))$ denote the corresponding multiplicative function.
\begin{enumerate}
[label=\text{(\alph*)}, ref=\text{(\alph*)}, leftmargin=*]
\item\label{thm8.1}
The additive function $a(n)$ is uniformly distributed in $\T$ if and only if
$\sum_{p\in\P}\frac1p\big(1-\Re(f^k(p) p^{it})\big)=\infty$ for all $t\in\R$ and all $k\geq 1$.
\item\label{thm8.2}
The additive function $a(n)$ has a limiting distribution $\nu$ that is not the Lebesgue measure if and only if there exists $k\in\N$ such that
$
\sum_{p\in\P}\frac{1}{p}\big(1-f^k(p)\big)
$
converges. The limiting distribution is continuous (i.e.\ the measure $\nu$ is non-atomic) if and only if
$$
\sum_{p\in \P\atop m a(p)\neq 0\bmod 1}\frac{1}{p}=\infty,\qquad\forall m\in\N.
$$
\end{enumerate}
\end{Theorem}

\cref{thm:Elliott-thm8.1-and-thm8.2} gives necessary and sufficient conditions for an additive function to have a limiting distribution.
In particular, if an additive function $a(n)$ satisfies neither condition \ref{thm8.1} nor condition \ref{thm8.2} of \cref{thm:Elliott-thm8.1-and-thm8.2} then $a(n)$ does not possess a limiting distribution. However, even in this case the limiting behavior of $a$ is well understood, as is demonstrated by \cref{thm:Elliott-thm8.9} below.
In order to formulate \cref{thm:Elliott-thm8.9}, it will be convenient to introduce first the following variant of \cref{def:nu-u.d.}.

\begin{Definition}
\label{def:nu-u.d.-2}
Let $\nu$ be a Borel probability measure on $\T$ and, for every $N\in\N$, let $x_N\colon \{1,\ldots,N\}\to\T$. Then $(x_N)_{N\in\N}$ is said to have \define{limiting distribution $\nu$} if for all continuous functions $F\in C(\T)$,
$$
\lim_{N\to\infty}\frac{1}{N}\sum_{n=1}^N F(x_N(n))=\int_\T F\, d\nu.
$$
\end{Definition}

\begin{Theorem}[see {\cite[Theorem 8.9]{Elliott79}}]\label{thm:Elliott-thm8.9}
Let $a\colon \N\to \T$ be an additive function. Then there exist $\alpha\colon \N\to\T$ and a Borel probability measure $\nu$ on $\T$ such that if $a_N\colon \{1,\ldots,N\}\to\T$ denotes the sequence
$$
a_N(n):=a(n)-\alpha(N),\qquad 1\leq n\leq N,
$$ 
then $(a_N)_{N\in\N}$ has a limiting distribution $\nu$.
Moreover, the measure $\nu$ is continuous (i.e.\ non-atomic) if and only if
$$
\sum_{p\in \P\atop m a(p)\neq 0 \bmod 1}\frac{1}{p}=\infty, \qquad\forall m\in\N.
$$
\end{Theorem}

\subsection{Additive functions with values in $\R$}\label{sec:distribution-R-valued-additive-functions}

In this subsection we summarize some known results regarding the distribution of real-valued additive functions.

Recall from \cref{ftnt-1} that an arithmetic function $a\colon \N\to\R$ is called \define{additive} if $a(n\cdot m)=a(n)+a(m)$ for all $m,n$ with $\gcd(n,m)=1$. For every additive function $a\colon \N\to\R$, the function
$$
f(n):=e^{a(n)}
$$
is a real-valued multiplicative function.

\begin{Definition}
\label{def:nu-u.d.-on-R}
Let $\nu$ be a Borel probability measure on $\R$.
A sequence $x\colon \N\to\R$ has \define{limiting distribution $\nu$} if for all bounded continuous functions $F\in C_{b}(\R)$,
$$
\lim_{N\to\infty}\frac{1}{N}\sum_{n=1}^N F(x(n))=\int_\R F\, d\nu.
$$ 
\end{Definition}

\begin{Theorem}[Erd{\H o}s-Wintner, see {\cite[Theorem 5.1]{Elliott79}}]\label{thm:Elliott-thm5.1}
An additive function $a\colon \N\to \R$ possess a limiting distribution if and only if the three series
$$
\sum_{p\in\P\atop |a(p)|>1}\frac{1}{p},\qquad\qquad
\sum_{p\in\P\atop |a(p)|\leq1}\frac{a(p)}{p},\qquad\qquad
\sum_{p\in\P\atop |a(p)|\leq1}\frac{(a(p))^2}{p}
$$
converge. In this case the corresponding measure is continuous (i.e.\ non-atmonic) if and only if
$$
\sum_{p\in\P\atop |a(p)|> 0}\frac{1}{p}=\infty.
$$
\end{Theorem}

\begin{Corollary}
\label{cor:Elliott-thm5.1}
Let $f\in\cm$ be a multiplicative function taking values in $(0,1]$ and assume $\|f\|_1\neq 0$. Then $f(n)$ possesses a limiting distribution.
This limiting distribution is continuous (i.e.\ the corresponding measure $\nu$ is non-atomic) if and only if $\sum_{p\in\P\atop f(p)\neq 1}\tfrac{1}{p}=\infty$.
\end{Corollary}

\begin{proof}
Let $a\colon \N\to\R$ denote the additive function $a(n):=\log(f(n))$. Note that $f$ has a limiting distribution if and only if $a$ has one.

We have $|a(p)|>1$ if and only if $f(p)\in\left(0,\tfrac{1}{e}\right)$. Since $\|f\|_1\neq 0$, it follows from \cref{lem:lH2-C} that $\sum_{p\in\P}\frac1p\big(1-f(p)\big)<\infty$. Therefore
$$
\sum_{p\in\P\atop |a(p)|>1}\frac{1}{p}~=~
\sum_{p\in\P\atop f(p)\in \left(0,e^{-1}\right)}\frac{1}{p}~\leq~\frac{e}{e-1}\sum_{p\in\P}\frac1p\big(1-f(p)\big)~<~\infty.
$$
Also, using the basic inequality $\tfrac{1}{e}(1-x)\geq -\log(x)$ for all $x\in \left[\tfrac{1}{e},1\right]$, we obtain
\begin{eqnarray*}
\sum_{p\in\P\atop |a(p)|\leq1}\frac{(a(p))^2}{p}
&\leq&
\left|\sum_{p\in\P\atop |a(p)|\leq1}\frac{a(p)}{p}\right|
\\
&=&
\sum_{p\in\P\atop f(p)\in \left[e^{-1},1\right]}\frac{-\log(f(p))}{p}
\\
&\leq&
\frac{1}{e}\left(\sum_{p\in\P\atop f(p)\in \left[e^{-1},1\right]}\frac{1}{p}\left(1-f(p)\right)\right)
\\
&\leq&\frac{1}{e}\left(\sum_{p\in\P}\frac1p\big(1-f(p)\big)\right)^2.
\end{eqnarray*}
Therefore, the three series
$$
\sum_{p\in\P\atop |a(p)|>1}\frac{1}{p},\qquad\qquad
\sum_{p\in\P\atop |a(p)|\leq1}\frac{a(p)}{p},\qquad\qquad
\sum_{p\in\P\atop |a(p)|\leq1}\frac{(a(p))^2}{p}
$$
converge and hence $a(n)$ possesses a distribution. Clearly, $f$ possesses a continuous distribution if and only if $a$ does, which is the case (by \cref{thm:Elliott-thm5.1}) if and only if
$\sum_{p\in\P\atop |a(p)|> 0}\frac{1}{p}=\sum_{p\in\P\atop f(p)\neq 1}\tfrac{1}{p}=\infty$.
\end{proof}

\section{Extending the K{\'a}tai orthogonality criterion}
\label{sec:Katai-all}
\label{sec:a(f)-and-proof-of-thmA}
\label{sec:KOC}

In \cref{sec:intro} we introduced the classes $\Done{\infty}$, $\DtwoA$ and $\DtwoB$; the statement of \cref{thm:Katai-for-sets} holds for any set $E$ belonging to either one of these two classes. In this section we will state and prove a generalization of \cref{thm:Katai-for-sets} where $\Done{\infty}$, $\DtwoA$ and $\DtwoB$ are replaced by the more general classes $\Dthree{\infty}$ and $\Dfour$ defined in the next subsection. This generalization is given by \cref{thm:Katai-for-sets-2} formulated in Subsection \ref{sec:thmN}.

\subsection{Definition of $\Dthree{\infty}$ and $\Dfour$}
\label{sec:a(f)}

Let $r\in\N$. A function $\vec{f}=(f_1,\ldots,f_r)\colon\N\to\C^r$ is called \define{multiplicative} if each of its coordinate components $f_i\colon\N\to\C$ is a multiplicative function. 
In accordance with the definition of concentration points for multiplicative functions $f\colon\N\to\C$ (cf. \cref{def:c1c2}), we say that a point $\vec{z}\in \C^r$ is a \define{concentration point} for a multiplicative function $\vec{f}\colon\N\to\C^r$ if the set $P:=\{p\in\P: \vec{f}(p)=\vec{z}\}$ satisfies $\sum_{p\in P}\frac{1}{p}=\infty$.

\begin{Definition}
\label{def:c3}
We denote by $\Dthree{r}$ the collection of all sets $E\subset \N$ of the form
$$
E(\vec{f},K):=\{n\in\N:\vec{f}(n)\in K\},
$$
where $K$ is an arbitrary subset of $\C^r$ and $\vec{f}\colon\N\to\C^r$ is a multiplicative function possessing at least one concentration point. Observe that $\DtwoA=\Dthree{1}$ and $\Dthree{i}\subset \Dthree{j}$ for $i\leq j$. We define $\Dthree{\infty}:=\bigcup_{r=1}^{\infty}\Dthree{r}$.
\end{Definition}

\begin{Proposition}
\label{prop:e-sets-1}
If $E\in \Done{r}$ and $d(E)>0$ then $E\in \Dthree{r}$.
\end{Proposition}

A proof of \cref{prop:e-sets-1} will be given in Subsection \ref{sec:proof-of-e-sets-1}.

In order to introduce the class $\Dfour$ we need the following definition.

\begin{Definition}\label{def:n(f)}
Let $f\colon\N\to\C$ be an arithmetic function.
We define $\cn(f)$ -- the class of \define{$f$-null sets} -- to be the collection of all sets $C\subset\C\setminus\{0\}$ such that
for all $\epsilon>0$ there exists a continuous function $F\colon \C\to[0,1]$ satisfying $F(z)=1$ for all $z\in C$ and
$$
\limsup_{N\to\infty}\frac{1}{N}\sum_{1\leq n\leq N\atop f(n)\neq 0} F(f(n))\leq \epsilon.
$$
\end{Definition}

In many cases multiplicative functions have a limiting distribution corresponding to a Borel probability measure $\nu$ (cf.\ Subsections \ref{sec:distribution-additive-functions} and \ref{sec:distribution-R-valued-additive-functions}). If this is the case then the class of $f$-null sets coincides with the class of $\nu$-null sets, i.e.\ all sets $C$ that satisfy $\nu(C)=0$.
For instance, if $f=\lio_{\xi}$ for some irrational $\xi\in\T$, then $(\lio_{\xi}(n))_{n\in\N}$ is uniformly distributed in the unit circle $\mathbb{S}^1:=\{z\in\C:|z|=1\}$ (by \cref{thm:Elliott-thm8.1-and-thm8.2} part \ref{thm8.1}). It is then straightforward to verify that a set $C\subset \C$ belongs to $\cn(\lio_{\xi})$ if and only if $C\cap \mathbb{S}^1$ has zero measure with respect to the Lebesgue measure on $\mathbb{S}^1$.

In the following let $\partial J:= \overline{J}\setminus J^\circ$ denote the boundary of a set $J\subset\C$.

\begin{Definition}
\label{def:afc3}
\
\begin{enumerate}
[label=\text{(\alph*)}, ref=\text{(\alph*)}, leftmargin=*]
\item
Given a multiplicative function $f$ define $\ca^*(f):=\{J\subset \C\setminus\{0\}:\partial J\in\cn(f)\}$ and
$$
\ca(f):= \ca^*(f)\cup \{J\cup\{0\}: J\in \ca^*(f)\}.
$$
It is straightforward to check that both $\ca^*(f)$ and $\ca(f)$ are algebras, i.e.~they are closed under finite unions, finite intersections and taking complements. 
\item
We denote by $\Dfour$ the collection of all sets $E\subset \N$ of the form $E(f,K):=\{n\in\N:f(n)\in K\}$, where $f\in\cm$ with $\|f\|_1\neq 0$, and $K\in\ca(f)$.
\end{enumerate}
\end{Definition}

\begin{Proposition}
\label{prop:e-sets-2}
We have $\DtwoB\subset \Dfour\cup \Dthree{1}$.
\end{Proposition}

A proof of \cref{prop:e-sets-2} is given in Subsection \ref{sec:proof-of-thmA}.

We will introduce and discuss now two pertinent families of general examples of sets belonging to $\Dthree{\infty}$ and/or $\Dfour$.

\begin{Example}\label{example:mf-d3d4}
\
\begin{enumerate}
[label=\text{Ex.\ref{example:mf-d3d4}.\arabic*:}, ref=\text{Ex.\ref{example:mf-d3d4}.\arabic*}, leftmargin=*]
\item
\label{item:eg:multiplicative-fiber-2-d3d4}
Let $\alpha_1,\ldots,\alpha_t,\beta_1,\ldots,\beta_t$ be real numbers and let $J_1,\ldots,J_t,I_1,\ldots,I_t$ be arbitrary subsets of $[0,1)$.
Consider the set
$$
E:=\{n\in\N: \Omega(n)\alpha_i\bmod 1 \in J_i~\text{and}~\omega(n)\beta_i\bmod 1 \in I_i~\text{for all}~i\in\{1,\ldots,t\}\}.
$$
Then $E$ belongs to the class $\Dthree{2t}$ because it can be written as
\begin{eqnarray*}
E
&=&\{n\in\N: \lio_{\alpha_i}(n) \in J_i'~\text{and}~\boldsymbol{\kappa}_{\beta_1} \in I_i'~\text{for all}~i\in\{1,\ldots,t\}\},
\end{eqnarray*}
where $\lio_{\xi}$ and $\boldsymbol{\kappa}_\xi$ are as defined in \ref{item:eg:multiplicative-functions-5} and $J_i':=\{e(x):x\in J_i\}$ and $I_i':=\{e(x):x\in I_i\}$.
Similarly, one can show that the sets $S_{\Omega,b_1,r_1}$, $S_{\omega,b_2,r_2}$, $S_{\omega,b_1,r_1}\cap S_{\Omega,b_2,r_2}$, $S_{\Omega,\alpha,J}$ and $S_{\omega,\alpha,J}$ from \cref{example:mf} belong to $\Dzero$, $\Done{2}$ and $\DtwoA$ respectively; in particular, they all belong to $\Dthree{\infty}$.
\item\label{item:eg:multiplicative-fiber-6-d3d4}
Let $f\colon \N\to\N$ be a multiplicative function and let $b,r\in\N$ with $\gcd(b,r)=1$. Let $t$ denote the number of generators of $(\Z/b\Z)^*$. We claim that the set
$$
E:=\{n\in\N: f(n)\equiv r\bmod b\}
$$
belongs to $\Done{t}$.
For the proof of this claim, choose $b_1,b_2,\ldots,b_t\in\N$ with $b=b_1\cdot\ldots\cdot b_t$ and such that
$(\Z/b\Z)^*$ is isomorphic to $C_{b_1}\times\ldots\times C_{b_t}$, where $C_n$ denotes the finite cyclic group of order $n$.
For $i\in\{1,\ldots,t\}$ let $c_i$ denote a generator of $C_{b_i}$.
We can identify $r$ with an element $(c_1^{r_1},\ldots,c_t^{r_t})\in C_{b_1}\times\ldots\times C_{b_t}$, where $r_i\in\{0,1,\ldots,b_i-1\}$ for all $i\in\{1,\ldots,t\}$. For $i\in\{1,\ldots,t\}$ define $\tilde\chi_i\colon  C_{b_1}\times\ldots\times C_{b_t}\to\C$ as
$$
\tilde\chi_i\left(c_1^{s_1},\ldots,c_t^{s_t}\right):=e\left(\frac{s_i}{b_i}\right).
$$
Then $\tilde\chi_i$ can be identified with a Dirichlet character $\chi_i$ of modulus $b$ via the isomorphism $(\Z/b\Z)^*\cong C_{b_1}\times\ldots\times C_{b_t}$. It is clear that
$$
\{n\in\N: n\equiv r\bmod b\}=\left\{n\in\N: \chi_1(n)=e\left(\frac{r_1}{b_1}\right),\ldots,\chi_t(n)=e\left(\frac{r_t}{b_t}\right)\right\}
$$
and therefore
$$
E=\left\{n\in\N: \chi_1(f(n))=e\left(\frac{r_1}{b_1}\right),\ldots,\chi_t(f(n))=e\left(\frac{r_t}{b_t}\right)\right\}.
$$
This proves that the set $E$ belongs to $\Done{t}$. In particular, by choosing $f=\boldsymbol{\tau}$, we see that the set $S_{\boldsymbol{\tau},b,r}$ from \ref{item:eg:multiplicative-fiber-6} belongs to $\Done{t}$.
\end{enumerate}
\end{Example}

\subsection{A generalization of \cref{thm:Katai-for-sets}}
\label{sec:thmN}

In light of Propositions \ref{prop:e-sets-1} and \ref{prop:e-sets-2} it is clear that the following result is a generalization of \cref{thm:Katai-for-sets}.

\begin{Theorem}
\label{thm:Katai-for-sets-2}
\label{thm:Katai-multipliactive-fibers-2}
Let $a\colon \N\to\C$ be a bounded sequence satisfying
\begin{equation}
\label{eq:KOC-a}
\sum_{n\leq x} a(pn)\overline{a(qn)} = \oh(x),~\text{for all $p,q\in\P$ with $p\neq q$.}
\end{equation}
Then for all sets $E\subset\N$ belonging to either $\Dthree{\infty}$ or $\Dfour$ we have
\begin{equation}
\label{eq:kc-conclusion-sets-a}
\sum_{n\leq x} \1_{E}(n)a(n)=\oh(x).
\end{equation}
\end{Theorem}

For the proof of \cref{thm:Katai-for-sets-2} we will need the following proposition.

\begin{Proposition}
\label{prop:finding-K_1-and-K_2}
Let $E\subset\N$ be a set that belongs to either $\Dthree{\infty}$ or $\Dfour$ and suppose $\overline{d}(E)>0$.
Then for all $\epsilon>0$ there exist sets $E_1\subset E_2\subset \C$ and a subset of prime numbers $P\subset \P$ satisfying:
\begin{enumerate}
[label=\text{(\roman*)}, ref=\text{(\roman*)}, leftmargin=*]
\item
$\overline{d}(E_2\setminus E_1)\leq \epsilon$;
\item
$\sum_{p\in P}\tfrac{1}{p}=\infty$;
\item
for all $p\in P$ and $n\in\N$ with $\gcd(n,p)=1$ we have $\1_{E_1}(n)\leq \1_{E}(n p)\leq \1_{E_2}(n)$.
\end{enumerate}  
\end{Proposition}

A proof of \cref{prop:finding-K_1-and-K_2} can be found in Subsection \ref{sec:finding-K_1-and-K_2}.

Another key ingredient for proving \cref{thm:Katai-multipliactive-fibers-2} is the following generalization of the K{\'a}tai Orthogonality Criterion (\cref{thm:org-Katai}), which we believe is of independent interest.

\begin{Proposition}\label{prop:pKOc}
Let $P_y$ be a subset of $\P$ with
$
p\leq y
$
for all
$
p\in P_y
$
and
\begin{equation}
\label{eqn:pKOc-1}
\sum_{p\in P_y} \frac{1}{p}
~~\xrightarrow{y\rightarrow\infty}~\infty.
\end{equation}
If $\s$, $\t_1$, $\t_2$ and $\rfctn$ are bounded real-valued arithmetic functions
such that for all $n\in \N$ and $p\in\bigcup_y P_y$ with $\gcd(n,p)=1$ one has
\begin{equation}
\label{eqn:pKOc-2}
\t_1(n)\rfctn(p)\leq\s(np)\leq\t_2(n)\rfctn(p)
\end{equation}
and if
$(u_n)$ is a bounded sequence in a Hilbert space $\Hilb$ satisfying
\begin{equation}
\label{eqn:pKOc-3}
\sum_{n\leq x} \la u_{pn}, u_{q n}\ra = \oh(x)
\end{equation}
for all $p,q\in \bigcup_y P_y$ with $p\neq q$ then
\begin{equation}
\label{eqn:pKOc-4}
\left\|\sum_{n\leq x} \s(n) u_n\right\| =\oh(x)+\Oh(x\|\t_1-\t_2\|_1).
\end{equation}
\end{Proposition}

A proof of \cref{prop:pKOc} is given in Subsection \ref{appendix:KOC}.

At this point we have collected all the tools needed to provide a proof of
\cref{thm:Katai-multipliactive-fibers-2}.

\begin{proof}[Proof of \cref{thm:Katai-multipliactive-fibers-2}]
Let $a(n)$ be a bounded sequence of complex numbers satisfying \eqref{eq:KOC-a}.
Let $E\subset\N$ be a set that belongs to either $\Dthree{\infty}$ or $\Dfour$. If $\overline{d}(E)=0$ then  \eqref{eq:kc-conclusion-sets-a} is trivially satisfied. Hence we can assume without loss of generality that $\overline{d}(E)>0$. Let $\epsilon>0$ be arbitrary.
According to \cref{prop:finding-K_1-and-K_2} there exist sets $E_1\subset E_2\subset \C$ and a set of prime numbers $P\subset \P$ satisfying $\overline{d}(E_2\setminus E_1)\leq \epsilon$, $\sum_{p\in P}\tfrac{1}{p}=\infty$, and $\1_{E_1}(n)\leq \1_{E}(n p)\leq \1_{E_2}(n)$ for all $p\in P$ and $n\in\N$ with $\gcd(n,p)=1$.

Now take $P_y:=P\cap[1,y]$, $F:=\1_E$, $G_1=\1_{E_1}$, $G_2=\1_{E_2}$, $H=1$ and $u_n=a(n)$. It follows immediately from $\overline{d}(E_2\setminus E_1)\leq \epsilon$ that $\|G_1-G_2\|_1\leq \vep$. Also, if $p\in P$ and $\gcd(n,p)=1$, then
$G_1(n) H(p)\leq F(np)\leq G_2(n) H(p)$. This means we can apply \cref{prop:pKOc} to obtain
\begin{equation}
\label{eqn:tKO-jordan-2}
\left|\sum_{n\leq x} F(n) u_n\right|=\left|\sum_{n\leq x} \1_{E}(n) a(n)\right| =\oh(x)+\Oh(x\vep).
\end{equation}
Since $\vep>0$ was chosen arbitrarily, this proves the theorem.
\end{proof}

We end this subsection with formulating an open question.

\begin{Question}
Consider the class $\DtwoC$ of all sets of the form $E(f,K):=\{n\in\N:f(n)\in K\}$, where $f\in\cm$ with $\|f\|_1> 0$ and $K$ is a Jordan measurable subset of $\C$. Observe that $\DtwoB\subset \DtwoC$.
Can \cref{thm:Katai-for-sets} be extended to the class $\DtwoC$?
\end{Question}

\subsection{Proof of \cref{prop:e-sets-1}}
\label{sec:proof-of-e-sets-1}

Before embarking on the proof of \cref{prop:e-sets-1} we need to define and discuss the notion concentrated multiplicative functions (which was introduced by Rusza in \cite{Ruzsa77}).

\begin{Definition}[cf.\ {\cite[Definition 3.8 and 3.9]{Ruzsa77}}]\label{def:concentration-pt-old}
\label{def:ST-concentrated-functions}
A multiplicative function $f\colon \N\to\C\setminus\{0\}$ is called \define{concentrated} if it satisfies
\begin{enumerate}
[label=\text{(\roman*)}, ref=\text{(\roman*)}, leftmargin=*]
\item\label{itm_a}
$f$ has at least one concentration point;
\item\label{itm_b}
the subgroup of $(\C\setminus\{0\},\cdot)$ generated by all concentration points of $f$, which we denote by $\G$, is finite; and
\item\label{itm_c}
$
\sum_{\substack{p\in\P,\\ f(p)\notin\G}}\frac{1}{p}<\infty.
$
\end{enumerate}
\end{Definition}

\begin{Theorem}[special case of {\cite[Theorem 3.10]{Ruzsa77}}]\label{thm:ST-Ruzsa-3.10}
Let $f:\N\to \C\setminus\{0\}$ be a multiplicative function. If $f$ is not concentrated then for all $z\in \C\setminus\{0\}$ the level set $E(f,z)$ has zero density.
\end{Theorem}

\begin{Corollary}[see {\cite[Corollary 2.17]{BKLR17arXiv}}]
\label{cor:density->concentrated function}
Let $f\colon \N\to\C$ be a multiplicative function and $z\in\C\setminus\{0\}$. If $d(E(f,z))>0$ then there exists a concentrated multiplicative function $g\colon \N\to\C\setminus\{0\}$ such that
$$
E(f,z)=E(g,z).
$$
\end{Corollary}


Before giving the proof of \cref{prop:e-sets-1} we need the following elementary lemma.
\begin{Lemma}
\label{lem:harmonic-sets}
Let $f_1,\ldots,f_r\colon \N\to\C$ be multiplicative functions and suppose that for every $i\in\{1,\ldots,r\}$ there exists a set of primes $P_i\subset\P$ satisfying the following two properties:
\begin{enumerate}
[label=\text{(\roman*)}, ref=\text{(\roman*)}, leftmargin=*]
\item
$\sum_{p\in\P\setminus P_i}\frac{1}{p}<\infty$;
\item
the set $\{f_i(p):p\in P_i\}$ is finite.
\end{enumerate}
Then there exist $z_1,\ldots,z_r\in\C$ and a set $P\subset\P$ with $\sum_{p\in P}\frac{1}{p}=\infty$ such that $f_i(p)=z_i$ for all $p\in P$ and all $1\leq i\leq r$.
\end{Lemma}

\begin{proof}
Let $P':=\bigcap_{i=1}^r P_i$. Then clearly $\sum_{p\in P}\frac{1}{p}=\infty$.
Moreover, $\{(f_1(p),\dots,f_r(p)):p\in P'\}$ is finite, so we get a finite partition of $P$ given by the possible $r$-tuples $(z_1,\dots,z_r)$ in the set $\{(f_1(p),\dots,f_r(p)):p\in P'\}$. By the pigeon hole principle, for at least one choice of $(z_1,\dots,z_r)$, the set $P=\{p\in P' : f_i(p)=z_i, 1\leq i\leq r\}$ satisfies $\sum_{p\in P}\frac{1}{p}=\infty$.
\end{proof}

\begin{proof}[Proof of \cref{prop:e-sets-1}]
Let  $E\in \Done{r}$ with $d(E)>0$ be given. By \cref{def:c1c2}, there exist multiplicative functions $f_1,\ldots,f_r\colon \N\to\C$ and complex numbers $z_1,\ldots,z_r$ such that $E=E(f_1,\ldots,f_r,z_1,\ldots,z_r)=\{n\in\N:f_1(n)=z_1,\ldots,~f_r(n)=z_r\}$.
Note that $E\subset E(f_i,z_i)=\{n\in\N:f_i(n)=z_i\}$, which implies that $d(E(f_i,z_i))>0$ for all $i\in\{1,\ldots,r\}$.

We now define new multiplicative functions $g_1,\ldots,g_r\colon \N\to\C$ in the following way:
For $i\in\{1,\ldots,r\}$, if $z_i=0$, set
$$
g_i(n):=
\begin{cases}
1,&\text{if $f_i(n)\neq 0$,}
\\
0,&\text{otherwise.}
\end{cases}
$$
On the other hand, if $z_i\neq0$, we take $g_i$ to be the concentrated multiplicative function guaranteed by \cref{cor:density->concentrated function}. Define $\vec{g}:=(g_1,\ldots,g_r)$, $\vec{z}:=(z_1,\ldots,z_r)$ and $K:=\{\vec{z}\}$.
Observe that
$$
E=E(\vec{g},K).
$$
It thus suffices to show that $E(\vec{g},K)\in\Dthree{r}$.

Note that for every $i\in\{1,\ldots,r\}$ there exists a set of primes $P_i\subset\P$, satisfying $\sum_{p\in\P\setminus P_i}\frac{1}{p}<\infty$, such that $\{g_i(p):p\in P_i\}$ is finite. In light of \cref{lem:harmonic-sets} we can find $w_1,\ldots,w_r\in\C$ and a set of primes $P\subset\P$ with $\sum_{p\in P}\frac{1}{p}=\infty$ such that $g_i(p)=w_i$ for all $p\in P$ and all $1\leq i\leq r$. This proves that $\vec{g}$ has a concentration point and hence $E(\vec{g},K)$ belongs to $\Dthree{r}$.
\end{proof}

\subsection{Proof of \cref{prop:e-sets-2}}
\label{sec:proof-of-thmA}

In this subsection we give a proof of \cref{prop:e-sets-2}.
First, we need the following useful lemma.

\begin{Lemma}
\label{lem:dilation-invariant-algebras-contained-in-a(f)}
Let $P\subset\P$ and assume $\sum_{p\in\P\setminus P}\tfrac{1}{p}<\infty$. Let $\ca$ be an algebra of subsets of $\C$ and suppose that for all $K\in\ca$ and all $u\in\C$ the set $uK$ belongs to $\ca$. Then for all $f,g\in\cm$ that satisfy $f(p)=g(p)$ for all $p\in P$ we have $\ca\subset \ca(f)$ if and only if $\ca\subset \ca(g)$.
\end{Lemma}

\begin{proof}
It follows from the definition of $\ca(f)$ that the set $K$ belongs to $\ca(f)$ if and only if $K\setminus\{0\}$ belongs to $\ca(f)$ (we will use this fact implicitly later).


Define the sets
\begin{equation}\label{eqn:S_P}
S_P:=\left\{n\in\N:\text{there exist distinct $p_1,\ldots, p_t\in P$ such that $n=p_1\cdot\ldots\cdot p_t$}\right\}
\end{equation}
and
\begin{equation}\label{eqn:T_P}
T_P:=\left\{n\in\N:\text{for all $p\in P$ if $p\mid n$ then $p^2\mid n$}\right\}.
\end{equation}
Note that the sets $S_P$ and $T_P$ are multiplicative, hence $\1_{S_P}$ and $\1_{T_P}$ are multiplicative functions (cf.\ \cref{footnote-p3}).
Also, $f\cdot\1_{S_P}=g\cdot\1_{S_P}$.

Since any natural number $n$ can be written uniquely as $st$, where $s\in S_P$, $t\in T_P$ and $\gcd(s,t)=1$, $\N$ can be partitioned into
\begin{equation}
\label{eq:S_P-T_P-partition}
\N=\bigcup_{t\in T_P}tS_P^{(t)},
\end{equation}
where $S_P^{(t)}:=\{s\in S_P:\gcd(s,t)=1\}$.

We now claim that for all $f\in\cm$, $\ca\subset\ca(f)$ if and only if $\ca\subset\ca(f\cdot\1_{S_P})$. Note that once we prove this claim, the proof of this lemma is completed, because $f\cdot\1_{S_P}=g\cdot\1_{S_P}$ and therefore $\ca\subset\ca(f)$ if and only if $\ca\subset\ca(g)$.

First, assume $\ca\subset\ca(f)$. Let $K\in \ca$ be arbitrary and let $J:=K\setminus\{0\}$.
Since $J\in \ca(f)$, for all $\epsilon>0$ there exists a continuous function $F\colon \C\to[0,1]$ such that $F(z)=1$ for all $z\in\partial J$ and
$$
\limsup_{N\to\infty}\frac{1}{N}\sum_{1\leq n\leq N\atop f(n)\neq 0} F(f(n))\leq \epsilon.
$$
This, however, implies
$$
\limsup_{N\to\infty}\frac{1}{N}\sum_{1\leq n\leq N\atop f(n)\cdot\1_{S_P}(n)\neq 0} F(f(n))\leq \epsilon,
$$
which shows that $J\in\ca(f\cdot\1_{S_P})$ and therefore $K\in\ca(f\cdot\1_{S_P})$.

Next, assume $\ca\subset\ca(f\cdot\1_{S_P})$. Again, let $K\in\ca$ be arbitrary. Fix $\epsilon>0$ and let $J:=K\setminus\{0\}$.
Note that $d(S_P)=M(\1_{S_P})$ exists (due to \cref{thm:wirsing}) and $d(S_P)>0$ because $\sum_{p\in\P\setminus P}\tfrac{1}{p}<\infty$ and therefore
$\sum_{p\in\P}\frac1p\big(1-\1_{S_P}(p)\big)<\infty$ (cf.\ \cref{lem:lH2-C}).
%
Likewise, $\1_{S_P^{(t)}}$ is a multiplicative function and hence $d(S_P^{(t)})=M(\1_{S_P}^{(t)})$ exists (again due to \cref{thm:wirsing}) and is positive (also by \cref{lem:lH2-C}). Using~\eqref{eq:S_P-T_P-partition} and the fact that $d(tS_P^{(t)})=t^{-1}d(S_P^{(t)})$ we obtain
\begin{equation}
\label{eq:S_P-T_P-1}
\sum_{t\in T_{P}}\tfrac{d(S_P^{(t)})}{t}=\sum_{t\in T_{P}}d(tS_P^{(t)})\leq d\left(\bigcup_{t\in T_{P}} tS_P^{(t)}\right) =d(\N)=1.
\end{equation}
For every $t\in T_P$ with $f(t)\neq 0$ the set $(f(t))^{-1}J\in\ca\subset\ca(f\cdot\1_{S_P})$. This means that for every $t\in T_P$ there exists a continuous function $F_t\colon \C\to[0,1]$ such that $F_t(z)=1$ for all $z\in \partial\big((f(t))^{-1} J\big)$ and
$$
\limsup_{N\to\infty}\frac{1}{N}\sum_{1\leq n\leq N\atop f\cdot\1_{S_P}(n)\neq 0} F_t(f(n)\cdot\1_{S_P}(n))\leq \frac{\epsilon d(S_P^{(t)})}{2}.
$$
Pick $M\geq 1$ sufficiently large such that $\sum_{t\in T_P\atop t>M }\tfrac{d(S_P^{(t)})}{t}\leq \tfrac{\epsilon}{2}$.
Define
$$
F(z):=\min_{t\in T_P\atop t\leq M}F_t\left((f(t))^{-1}z\right).
$$
Certainly, $F$ is continuous and $F(z)=1$ for all $z\in\partial J$.
Moreover,
\begin{align*}
\limsup_{N\to\infty}\frac{1}{N}&\sum_{1\leq n\leq N\atop f(n)\neq 0} F(f(n))
\\
&=
\limsup_{N\to\infty}
\frac{1}{N}\sum_{t\in T_P,\atop f(t)\neq 0}\left(\sum_{\substack{s\in S_P^{(t)}\cap \left[1,\tfrac{N}{t}\right] \\ f(s)\neq 0}} F(f(ts))\right)
\\
&\leq
\limsup_{N\to\infty}
\sum_{\substack{t\in T_P\\ t\leq M \\ f(t)\neq 0}}\left(
\frac{1}{N}\sum_{\substack{s\in S_P^{(t)}\cap \left[1,\tfrac{N}{t}\right] \\ f(s)\neq 0}} F(f(t)f(s))\right)~+~\sum_{t\in T_P\atop t>M }\frac{d(S_P^{(t)})}{t}
\\
&\leq
\sum_{\substack{t\in T_P\\ t\leq M \\ f(t)\neq 0}}
\frac{1}{t}\left(\limsup_{N\to\infty}
\frac{t}{N}\sum_{\substack{s\in S_P^{(t)}\cap \left[1,\tfrac{N}{t}\right] \\ f(s)\neq 0}} F_t(f(s))\right)~+~\frac{\epsilon}{2}
\\
&\leq
\sum_{\substack{t\in T_P\\ t\leq M \\ f(t)\neq 0}}
\frac{1}{t}\left(\limsup_{N\to\infty}
\frac{t}{N}\sum_{\substack{1\leq s\leq \tfrac{N}{t} \\ f\cdot\1_{S_P}(s)\neq 0}} F_t(f(s)\cdot\1_{S_P}(s))\right)~+~\frac{\epsilon}{2}
\\
&\leq
\frac{\epsilon}{2}\sum_{t\in T_P}
\frac{d(S_P^{(t)})}{t}~+~\frac{\epsilon}{2}
~\leq~
\epsilon.
\end{align*}
Since $\epsilon>0$ was arbitrary, we conclude that $J\in\ca(f)$ and therefore $K\in\ca(f)$.
\end{proof}

Let $\Vert x\Vert$ denote the distance of a real number $x$ to the closest integer.
For every $\delta>0$ and every $y\in \T$ define function $F_{y,\delta}\in C(\T)$ as
\begin{equation}
\label{eq:F_ye}
F_{y,\delta}(x):=
\begin{cases}
1-\frac{\Vert x-y\Vert}{\delta},&\text{if}~\Vert x-y\Vert\leq \delta;
\\
0,&\text{otherwise.}
\end{cases}
\end{equation}

\begin{Lemma}
\label{lem:continuous-measure-weak-limit}
Let $\nu$ be a Borel probability measure on $\T$ and let $(\nu_N)_{N\in\N}$ be a sequence of Borel probability measures on $\T$ that converges to $\nu$ in the weak-*-topology (i.e., for all $F\in C(\T)$,
$
\lim_{N\to\infty} \int_\T F\, d\nu_N= \int_\T F\, d\nu
$).
If $\nu$ is non-atomic then for every $\epsilon>0$ there exist $\delta>0$ and $N_0\in\N$ such that
$$
\int_\T F_{y,\delta}\, d\nu_N < \epsilon
$$
for all $y\in\T$ and for all $N\geq N_0$.
\end{Lemma}

\begin{proof}
Define $I_\delta(y):=\int_\T F_{y,\delta}\, d\nu$. It is clear that $I_\delta$ is a continuous function on $\T$ for every $\delta\in(0,1)$. Also, the family $(I_\delta)_{\delta\in(0,1)}$ is monotonically decreasing in the sense that $I_{\delta_1}(y)\geq I_{\delta_2}(y)$ for all $y\in\T$ and all $\delta_1\geq\delta_2\in(0,1)$. Since $\nu$ is non-atomic, the functions $F_{y,\delta}(x)$ converge to $0$ for $\nu$-almost every $x$. Therefore, by the monotone convergence theorem, $I_\delta(y)$ converges to $0$ as $\delta\to0$ for every $y$.

We invoke now the classical Dini theorem, which states that a monotonically decreasing sequence of continuous real-valued functions that converges pointwise to a continuous function convergences uniformly. Therefore $I_\delta$ converges to $0$ uniformly as $\delta\to0$.

Fix now some $\epsilon>0$. Pick $\delta>0$ such that $\sup_{y\in\T} I_{2 \delta}(y)<\tfrac{\epsilon}{2}$. We claim that there exists $N_0$ such that for all $N\geq N_0$ and all $y\in\T$ we have
$$
\int_{\T} F_{y,\delta}\, d\nu_N <\epsilon.
$$

Assume that, contrary to our claim, there exists an increasing sequence of natural numbers $(N_j)_{j\in\N}$ such that for every $j\in \N$ there exists $y_j\in\T$ with
$$
\int_{\T} F_{y_j,\delta} \, d\nu_{N_j} \geq\epsilon.
$$
The sequence $(y_j)_{j\in\N}$ has a convergent subsequence. Hence, by passing to it if necessary, we can assume without loss of generality that $\lim_{j\to\infty} y_j$ exists. Let $y\in\T$ denote this limit.
It is straightforward to verify that for sufficiently large $j$ we have
$$
F_{y_j,\delta}(x)\leq 2F_{y,2\delta}(x),\qquad\forall x\in\T.
$$ 
Therefore, 
\begin{eqnarray*}
\limsup_{j\to\infty} \int_{\T} F_{y_j,\delta} \, d\nu_{N_j}
&\leq &
\limsup_{j\to\infty} \int_{\T} 2F_{y,2\delta} \, d\nu_{N_j}
\\
&= &
\int_{\T} 2F_{y,2\delta} \, d\nu
\\
&\leq &
2\sup_{y\in\T} I_{2 \delta}(y)
\\
&<&\epsilon.
\end{eqnarray*}
This contradicts $\int_{\T} F_{y_j,\delta} \, d\nu_{N_j} \geq\epsilon$ for all $j\in\N$.
\end{proof}

\begin{Lemma}\label{lem:lH-balls-at-zero-new}
Suppose $f\in\cm$ satisfies $\|f\|_1\neq 0$ and $f(n)\neq 0$ for all $n\in\N$. 
Then
\begin{equation}
\label{en:lH-ball-at-zero-11}
\lim_{\vep\to 0}\overline{d}\big(\{n\in\N: |f(n)|<\epsilon\}\big)= 0.
\end{equation}
\end{Lemma}

The following proof of \cref{lem:lH-balls-at-zero-new} was provided by a user with alias Lucia as an answer to a question posted by the third author at http://mathoverflow.net. We gratefully acknowledge Lucia's help. 


\begin{proof}[Proof of \cref{lem:lH-balls-at-zero-new} (see {\scriptsize \href{http://mathoverflow.net/questions/215170}{http://mathoverflow.net/questions/215170}})]
By replacing $f$ with $\vert f\vert$ if necessary, we can ssume without loss of generality that $f$ takes values in $(0,1]$.
For $0<\delta<1$ and $k\geq 1$ put 
\[
F_k(\delta):= \sum_{p\in\P,\atop f(p^k) \leq \delta} \frac{1}{p^k}\quad\text{and}\quad
F(\delta):=\sum_{k=1}^\infty F_k(\delta).
\]
Since $\|f\|_1>0$, it follows from \cref{lem:lH2-C} that $\sum_{p\in\P}\frac1p\big(1-f(p)\big)<\infty$.
This shows that $F_1(\delta)<\infty$ for every $0<\delta<1$ and so $F_k$ is a well defined function for all $k\geq 1$.
Moreover, since $\sum_{k\geq 2}\sum_{p\in\P}\frac{1}{p^k}=\sum_{p\in\P}\frac{1}{p(p-1)}<\infty$, the function $F$ is well defined in $(0,1)$.

We claim that $F(\delta)$ converges to zero as $\delta\to0$.  For $0<\delta<1$, let 
$$
\sB_\delta:=\{p^k: p\in\P, k\in\N, f(p^k)\leq \delta\}.
$$
We have $F(\delta)=\sum_{p^k\in \sB_\delta}\frac{1}{p^k}<\infty$. In particular, $F(1/2)<\infty$ and there exists a finite set $H\subset\sB_{\frac{1}{2}}$ such that $\sum_{p^k\in\sB_{1/2}\setminus H}\frac{1}{p^k}\leq \vep$. Take $0<\delta<\min_{p^k\in H} f(p^k)$. Then $\sB_\delta\subset \sB_{1/2}\setminus H$ and therefore $F(\delta)\leq \sum_{p^k\in\sB_{1/2}\setminus H}\frac{1}{p^k}\leq \vep$.


For $0<\delta<1$, let $\cf_{\sB_\delta}$ denote the set of $\sB_\delta$-free numbers, that is $\cf_{\sB_\delta}:=\N\setminus\left(\bigcup_{p^k\in \sB_\delta} p^k\N\right)$. It is straightforward to show that
$$
d(\cf_{\sB_\delta})=1-d\left(\bigcup_{p^k\in \sB_\delta} p^k\N\right)\geq 1-\sum_{p^k\in \sB_\delta}\frac{1}{p^k}=1-F(\delta).
$$
So,
\begin{equation*}
\overline{d}\big(\{n\in\N: f(n)<\epsilon\}\big)\leq \overline{d}\big(\{n\in \cf_{\sB_\delta}: f(n)<\epsilon\}\big) + F(\delta).
\end{equation*}
Notice that $x\geq \exp\left(2\log(\delta)(1-x) \right)$ for any $x\in (\delta,1]$. Moreover, for $n=p_1^{k_1}\cdots p_r^{k_r}\in \cf_{\sB_\delta}$, $p_i^{k_i}\in \cf_{\sB_\delta}$ for $1\leq i \leq r$, so, in particular, $p_i^{k_i}\not\in \sB_\delta$ whence $f(p_i^{k_i})>\delta$, $1\leq i\leq r$. Thus, for each $n=p_1^{k_1}\cdots p_r^{k_r}\in \cf_{\sB_\delta}$, we have
\begin{equation}
\label{eqn:lH-ball-at-zero-21}
f(n)= f(p_1^{k_1})\cdots f(p_r^{k_r})\geq  \exp\left(2\log\left(\delta\right)\sum_{i=1}^k (1-f(p_i^{k_i}))\right).
\end{equation}
So, if $f(n)<\vep$ and $n\in \cf_{\sB_\delta}$, then \eqref{eqn:lH-ball-at-zero-21} implies that
\[
\sum_{p^k\mid n} (1-f(p^k))\geq \frac{\log(\vep)}{2\log(\delta)}.
\]
This shows that
\begin{eqnarray*}
\frac{1}{x}\big|\{n\leq x:n\in \cf_{\sB_\delta},~ f(n)<\epsilon\}\big|
&\leq&\frac{1}{x}\frac{2\log(\delta)}{\log(\vep)}\sum_{n\leq x}\sum_{p^k\mid n} (1-f(p^k))\\
&\leq&\frac{2\log(\delta)}{\log(\vep)}\sum_{p\in\P,\atop k\in\N}\frac{(1-f(p^k))}{p^k}\\
&=&\Oh\left(\frac{\log(\delta)}{\log(\vep)}\right).
\end{eqnarray*}
Finally, if we set $\delta=\exp(-\sqrt{-\log(\vep)})$, which goes to zero as $\vep$ goes to zero, then this shows that for $\vep>0$ sufficiently small
\begin{equation*}
\overline{d}\big(\{n\in\N: f(n)<\vep\}\big)=  \Oh\left(\frac{1}{\sqrt{-\log(\vep)}}\right)+ F\left(\exp(-\sqrt{-\log(\vep)})\right),
\end{equation*}
which completes the proof.
\end{proof}

We are now ready to give a proof of \cref{prop:e-sets-2}.

\begin{proof}[Proof of \cref{prop:e-sets-2}]
Suppose $E$ belongs to $\DtwoB$. This means that $E$ is of the form $E(f,K):=\{n\in\N:f(n)\in K\}$, where $f\in\cm$ with $\|f\|_1\neq0$ and $K$ is an elementary set in polar coordinates.
If $f$ has a concentration point then $E\in \Dthree{1}$ and we are done. Let us therefore assume that $f$ possesses no concentration points.
It remains to show that any elementary set in polar coordinates belongs to $\ca(f)$, because this implies that $E\in\Dfour$.

Let $f'\in\cm$ denote the multiplicative function uniquely determined by
\begin{equation*}
f'(p^k):=
\begin{cases}
f(p^k),&\text{if $f(p^k)\neq 0$,}
\\
1,&\text{otherwise.}
\end{cases}
\end{equation*} 
Let $P$ denote the set of all primes $p$ such that $f(p)= f'(p)$. Since $\|f\|_1\neq0$, it follows from \cref{lem:lH2-C} that $\sum_{p\in\P\setminus P}\tfrac{1}{p}<\infty$.
Therefore, using \cref{lem:dilation-invariant-algebras-contained-in-a(f)}, we deduce that $\ca(f)$ contains all elementary sets in polar coordinates if and only if $\ca(f')$ does. We can therefore assume without loss of generality that $f(n)\neq0$ for all $n\in\N$.

Recall that $e(x):=e^{2\pi i x}$. Now suppose $K:=\{re(\varphi): \varphi\in I_1, r\in I_2 \}$, where $I_1$ is a subinterval of $\T$ and $I_2$ is a subinterval of $[0,1]$.
We assume that both $I_1$ and $I_2$ are closed intervals and remark that for open and half-open intervals the same argument applies. Choose $a_1,b_1\in\T$ such that $I_1=[a_1,b_1]$ and $a_2,b_2\in[0,1]$ such that $I_2=[a_2,b_2]$.

Let $h(n):=|f(n)|$, $n\in\N$, and let $g(n):=\frac{f(n)}{|f(n)|}$.
Clearly, $f=g\cdot h$. 
Let $a\colon \N\to\T$ be the (unique) additive function such that $g(n)=e(a(n))$ for all $n\in \N$.

We now distinguish three cases:
\begin{enumerate}
[label=\text{(\roman*)}, ref=\text{(\roman*)}, leftmargin=*]
\item\label{itm:thmA-case-1}
$\sum_{p\in\P\atop h(p)\neq 1}\tfrac{1}{p}<\infty$ and $\sum_{p\in\P\atop m a(p)\neq 0\bmod 1}\tfrac{1}{p}<\infty$ for some $m\in\N$;
\item\label{itm:thmA-case-2}
$\sum_{p\in\P\atop h(p)\neq 1}\tfrac{1}{p}=\infty$;
\item\label{itm:thmA-case-3}
$\sum_{p\in\P\atop m a(p)\neq 0\bmod 1}\tfrac{1}{p}=\infty$ for all $m\in\N$.
\end{enumerate}

In case \ref{itm:thmA-case-1}, one of the $m$-th roots of unity is a concentration point of $f$, which contradicts the assumption that $f$ possesses no concentration points.
Therefore  we only have to deal with cases \ref{itm:thmA-case-2} and \ref{itm:thmA-case-3}.

In case \ref{itm:thmA-case-2}, $h(n)$ possesses a continuous limiting distribution given by a Borel probability measure $\nu_2$ on $[0,1]$ (cf.\ \cref{cor:Elliott-thm5.1}).
Let $\epsilon>0$ be arbitrary.
Pick a continuous  $F_2\colon \R\to[0,1]$ such that $F_2(a_2)=F_2(b_2)=1$ and $\int_0^1 F_2\, d\nu_2\leq \epsilon$; such a function is guaranteed to exist because $\nu_2$ is non-atomic. Define a new function $F\colon \C\to[0,1]$ as $F(re(\varphi))=F_2(r)$. Notice that $F(z)=1$ for all $z\in\partial K$. Moreover,
\begin{eqnarray*}
\limsup_{N\to\infty}\frac{1}{N}\sum_{n=1}^N F(f(n))
&=&\limsup_{N\to\infty}\frac{1}{N}\sum_{n=1}^N F(g(n)h(n))
\\
&=&\limsup_{N\to\infty}\frac{1}{N}\sum_{n=1}^N F_2(h(n))
\\
&=&
\int_0^1 F_2\, d\nu_2\leq \epsilon.
\end{eqnarray*}
Since $\epsilon>0$ was chosen arbitrarily, this proves that $K\in\ca(f)$.

Next, we deal with case \ref{itm:thmA-case-3}.
Using \cref{thm:Elliott-thm8.9} we can find $\alpha\colon \N\to\T$ and a probability measure $\nu$ on $\T$ such that if $a_N\colon \{1,\ldots,N\}\to\T$ denotes the sequence
$$
a_N(n):=a(n)-\alpha(N),\qquad 1\leq n\leq N,
$$ 
then $(a_N)_{N\in\N}$ has limiting distribution $\nu$. Moreover, this limiting distribution is continuous because $\sum_{p\in\P\atop m a(p)\neq 0\bmod 1}\tfrac{1}{p}=\infty$ for all $m\in\N$.
Fix $\epsilon>0$. For $y\in\T$ let $\delta_y$ denote the point-mass at $y$. Define
$$
\nu_N:=\frac{1}{N}\sum_{n=1}^N \delta_{a(n)-\alpha(N)}.
$$
By definition, the limit of $(\nu_N)_{N\in\N}$ in the weak-*-topology equals $\nu$.
Let $F_{y,\delta}$ be as defined in \eqref{eq:F_ye}.
Using \cref{lem:continuous-measure-weak-limit} we can find $\delta>0$ and $N_0\in\N$ such that
\begin{equation}\label{eq:b-o-int}
\int_\T F_{y,\delta}\, d\nu_N < \frac{\epsilon}{3}
\end{equation}
for all $y\in\T$ and for all $N\geq N_0$.
In view of \cref{lem:lH-balls-at-zero-new} we have
\begin{equation*}
\label{en:lH-ball-at-zero-11-new}
\lim_{\eta\to 0}\overline{d}\big(\{n\in\N: |f(n)|<\eta\}\big)= 0.
\end{equation*}
In particular, there exists $\eta>0$ such that
$$
\overline{d}\big(\{n\in\N: |f(n)|<\eta\}\big)<\frac{\epsilon}{3}.
$$
Let $\tilde{F}\colon\{re(\varphi): \varphi\in\T, r\in[\eta,1]\}\to[0,1]$ denote the function
$$
\tilde{F}(re(\varphi)):=\max\{F_{a_1,\delta}(\varphi),F_{b_1,\delta}(\varphi)\}.
$$
Let $F\colon \C\to [0,1]$ be an arbitrary continuous continuation of $\tilde{F}$ to all of $\C$ that satisfies $F(z)=1$ for all $z\in\partial K$.
Then
\begin{eqnarray*}
\limsup_{N\to\infty}\frac{1}{N}\sum_{n=1}^N F(f(n))
&=&\limsup_{N\to\infty}\frac{1}{N}\sum_{n=1}^N \Big(\1_{[|f|<\eta]}(n)F(f(n))+\1_{[|f|\geq\eta]}(n)F(f(n))\Big)
\\
&\leq&\limsup_{N\to\infty}\frac{1}{N}\sum_{n=1}^N \1_{[|f|\geq\eta]}(n)F(f(n)) +\frac{\epsilon}{3}
\\
&\leq&\limsup_{N\to\infty}\frac{1}{N}\sum_{n=1}^N \1_{[|f|\geq\eta]}(n) \tilde{F}(h(n)g(n)) +\frac{\epsilon}{3}
\\
&\leq&\limsup_{N\to\infty}\frac{1}{N}\sum_{n=1}^N F_{a_1,\delta}(a(n)) +\limsup_{N\to\infty}\frac{1}{N}\sum_{n=1}^N F_{b_1,\delta}(a(n)) +\frac{\epsilon}{3}.
\end{eqnarray*}
Now observe that
$$
\frac{1}{N}\sum_{n=1}^N F_{a_1,\delta}(a(n)) =
\frac{1}{N}\sum_{n=1}^N F_{a_1-\alpha(N),\delta}(a(n)-\alpha(N))
=\int_{\T} F_{a_1-\alpha(N),\delta}\, d\nu_N.
$$
It follows from \eqref{eq:b-o-int} that
$$
\limsup_{N\to\infty}\frac{1}{N}\sum_{n=1}^N F_{a_1,\delta}(a(n))\leq\frac{\epsilon}{3}.
$$
An analogous argument shows that
$$
\limsup_{N\to\infty}\frac{1}{N}\sum_{n=1}^N F_{b_1,\delta}(a(n))\leq\frac{\epsilon}{3}.
$$
We conclude that
\begin{eqnarray*}
\limsup_{N\to\infty}\frac{1}{N}\sum_{n=1}^N F(f(n))
&\leq&\limsup_{N\to\infty}\frac{1}{N}\sum_{n=1}^N F_{a_1,\delta}(a(n)) +\limsup_{N\to\infty}\frac{1}{N}\sum_{n=1}^N F_{b_1,\delta}(a(n)) +\frac{\epsilon}{3}
\\
&\leq &\frac{\epsilon}{3}+\frac{\epsilon}{3}+\frac{\epsilon}{3}=\epsilon.
\end{eqnarray*}
To summarize, the function $F\colon \C\to[0,1]$ is continuous, it satisfies $F(z)=1$ for all $z\in\partial K$ and it also satsifies
$$
\limsup_{N\to\infty}\frac{1}{N}\sum_{n=1}^N F(f(n))\leq \epsilon.
$$
Since $\epsilon>0$ is arbitrary, this proves that $K\in\ca(f)$.
\end{proof}

\subsection{Proof of \cref{prop:finding-K_1-and-K_2}}
\label{sec:finding-K_1-and-K_2}

The purpose of this subsection is to present a proof of \cref{prop:finding-K_1-and-K_2}. 
The proof of \cref{prop:finding-K_1-and-K_2} for the case $E\in\Dthree{\infty}$ is fairly easy and straightforward; the proof for the case $E\in\Dfour$, however, is more complicated and relies on the following lemma.

\begin{Lemma}
\label{lem:finding-K_1-and-K_2-for-D_4}
Let $f\in\cm$ with $\|f\|_1\neq 0$ and let $K\in\ca(f)$. Then for all $\epsilon>0$ there exist sets $K_1\subset K_2\subset \C$ and a set of prime numbers $P\subset \P$ satisfying:
\begin{itemize}
\item[--] $f(p)\neq 0$ for all $p\in P$;
\item[--] $\sum_{p\in P}\tfrac{1}{p}=\infty$;
\item[--] $f(p) K_1\subset K \subset f(p)K_2$ for all $p\in P$;
\item[--] $\overline{d}(\{n\in \N: f(n)\in K_2\setminus K_1\})\leq\epsilon $.
\end{itemize}  
\end{Lemma}

The proof of \cref{lem:finding-K_1-and-K_2-for-D_4} hinges on two other lemmas, namely Lemmas \ref{lem:harmonic-subset-of-primes} and \ref{lem:finding-K_1-and-K_2}, which we state and prove first.

\begin{Lemma}
\label{lem:harmonic-subset-of-primes}
Let $f\in\cm$ with $\|f\|_1\neq 0$. Then there exists $u\in\C$ with $|u|= 1$ such that for all $\delta>0$ the set $P_{u,\delta}:=\{p\in\P:|f(p)-u|<\delta\}$ satisfies $\sum_{p\in P_{u,\delta}}\tfrac{1}{p}=\infty$.
\end{Lemma}

\begin{proof}
Recall that $\mathbb{S}^1=\{z\in\C: |z|=1\}$.
Suppose that for every $u\in\mathbb{S}^1$ there exists some $\delta_u>0$ such that $\sum_{p\in P_{u,\delta_u}}\tfrac{1}{p}<\infty$. Since $B(u,\delta_u):=\{z\in\C: |u-z|<\delta_u\}$, $u\in\mathbb{S}^1$, is an open cover of the compact set $\mathbb{S}^1$, we can find a finite sub-cover. In other words, there exist $u_1,\ldots,u_r\in\C$, $|u_i|= 1$ for $i=1,\ldots,r$, such that $\bigcup_{i=1}^rB(u_i,\delta_{u_i})\supset\mathbb{S}^1$.
Since $\bigcup_{i=1}^rB(u_i,\delta_{u_i})$ is an open set containing $\mathbb{S}^1$, there exists some $\delta>0$ such that the set $\{z\in\C: 1-\delta <|z|<1+\delta\}$ is contained in $\bigcup_{i=1}^rB(u_i,\delta_{u_i})$. Define $P:=\{p\in\P: |f(p)|>1-\delta\}$.
Then we have
$$\sum_{p\in P}\frac{1}{p} ~\leq~ \sum_{i=1}^r \left(\sum_{p\in P_{u_i,\delta_{u_i}}}\frac{1}{p}\right)~<~\infty.$$
One the other hand, it follows from $\|f\|_1\neq 0$ and \cref{lem:lH2-C} that $\sum_{p\in\P}\frac1p\big(1-|f(p)|\big)<\infty$ and therefore
$$
\sum_{p\in \P\setminus P} \frac{1}{p}\leq \frac{1}{\delta}\sum_{p\in \P} \frac{1}{p}(1-|f(p)|)<\infty.
$$
However, $\sum_{p\in \P\setminus P} \frac{1}{p}<\infty$ and $\sum_{p\in P} \frac{1}{p}<\infty$ yield a contradiction.
\end{proof}



\begin{Lemma}
\label{lem:finding-K_1-and-K_2}
Let $f\in\cm$ with $\|f\|_1\neq 0$, let $J\subset\C\setminus\{0\}$ and assume that $\partial J\in\cn(f)$. Then for all $\epsilon>0$ there exist sets $J_1\subset J_2\subset \C\setminus\{0\}$ and a set of prime numbers $P\subset \P$ satisfying:
\begin{itemize}
\item[--] $f(p)\neq 0$ for all $p\in P$;
\item[--] $\sum_{p\in P}\tfrac{1}{p}=\infty$;
\item[--] $f(p)J_1\subset J\subset f(p)J_2$ for all $p\in P$;
\item[--] $\overline{d}(\{n\in \N: f(n)\in J_2\setminus J_1\})\leq\epsilon $.
\end{itemize}  
\end{Lemma}

\begin{proof}
Let $\epsilon>0$ be arbitrary and let $u\in\C$ be as guaranteed by \cref{lem:harmonic-subset-of-primes}.
We can find a continuous function $F\colon \C\to[0,1]$ satisfying $F(z)=1$ for all $z\in \partial J$ and
$$
\limsup_{N\to\infty}\frac{1}{N}\sum_{1\leq n\leq N\atop f(n)\neq 0} F(f(n))\leq \frac{\epsilon}{4}.
$$
Let $D:=\{z\in\C:|z|\leq 1\}$ be the unit disc in $\C$.
We define a new function $G\colon D\to[0,1]$ as $G(z)=F(uz)$ for all $z\in D$. Note that $G$ has the property that $G(z)=1$ for all $z\in \partial(\overline{u}J)$.
Let $S:=\left\{z\in\C\setminus\{0\}: G(z)>\tfrac12\right\}$ and define $J_1:= (\overline{u}J)\setminus S$ and $J_2:=(\overline{u}J)\cup S$. It remains to show that $J_1$ and $J_2$ have the desired properties.

Since $G$ is uniformly continuous, there exists some $\delta_0>0$ such that for all $z,w\in D$
\begin{equation}
\label{eq:unif-cont-of-G}
|z-w|<\delta_0\quad\implies\quad|G(z)-G(w)|\leq \min\left\{\tfrac{\epsilon}{4},\tfrac{1}{4}\right\}.
\end{equation}
Take $\delta:=\min\left\{\tfrac{\delta_0}{2},\tfrac12\right\}$.
We claim that
\begin{eqnarray}
\label{eqn:inc-J-1}
J_1+B(0,\delta)&\subset & \overline{u}J,
\\
\label{eqn:inc-J-2}
\overline{u}J+B(0,\delta)& \subset & J_2,
\end{eqnarray}
where $B(0,\delta):=\{z\in\C: |z|<\delta\}$.

We prove \eqref{eqn:inc-J-1} by contradiction. Assume there are $w\in J_1$ and $z\notin \overline{u}J$ such that $|w-z|<\delta$. Since $w\in \overline{u}J$ and $z\notin\overline{u}J$, there exists a point $y\in\partial(\overline{u}J)$ with $|w-y|< \delta$.
Using \eqref{eq:unif-cont-of-G} and the fact that $G(y)=1$ we deduce that $G(w)>\frac{3}{4}$. In particular, $w\in S$. However, this contradicts the fact that $J_1\cap S=\emptyset$.  The inclusion in \eqref{eqn:inc-J-2} can be proved in a similar way.

Let $P_{u,\delta}$ be as in the statement of \cref{lem:harmonic-subset-of-primes} and define
$P:=P_{u,\delta}$.
Then $\sum_{p\in P}\tfrac{1}{p}=\infty$.
Also, for all $p\in P$ we have $|f(p)-u|<\delta$ and therefore $f(p) J_1\subset u J_1+B(0,\delta)$. Using \eqref{eqn:inc-J-1}, we then obtain that $f(p) J_1\subset J$.
Analogously, using $|f(p)-u|<\delta$ and \eqref{eqn:inc-J-2} we get $J\subset f(p) J_2$ for all $p\in P$.

It remains to show that $\overline{d}(\{n\in \N: f(n)\in J_2\setminus J_1\})\leq\epsilon $. Take any $p\in P$ that satisfies $\tfrac{1}{p}<\tfrac{\epsilon}{4}$. Note that
\begin{eqnarray*}
\overline{d}(\{n\in \N: f(n)\in J_2\setminus J_1\})
&=&\overline{d}(\{n\in \N: f(n)\in S\})\\
&\leq& \limsup_{N\to\infty}\frac{2}{N}\sum_{1\leq n\leq N\atop f(n)\neq 0} G(f(n))\\
&=& \limsup_{N\to\infty}\frac{2}{N}\sum_{1\leq n\leq N\atop f(n)\neq 0} F(u f(n)).
\end{eqnarray*}
Using \eqref{eq:unif-cont-of-G} we get that $|F(u f(n))-F(f(p)f(n))|\leq \tfrac{\epsilon}{4}$. Hence,
$$
\overline{d}(\{n\in \N: f(n)\in J_2\setminus J_1\})
\leq  \limsup_{N\to\infty}\frac{2}{N}\sum_{1\leq n\leq N\atop f(n)\neq 0} F(f(p) f(n)) ~+~\frac{\epsilon}{4}.
$$
Finally,
\begin{eqnarray*}
\limsup_{N\to\infty}\frac{2}{N}\sum_{1\leq n\leq N\atop f(n)\neq 0} F(f(p) f(n))
&\leq & 
\limsup_{N\to\infty}\frac{2}{N}\sum_{\substack{1\leq n\leq N \\ \gcd(p,n)=1 \\ f(n)\neq 0}}F(f(p) f(n)) ~+~\frac{2}{p}
\\
&= & 
\limsup_{N\to\infty}\frac{2}{N}\sum_{\substack{1\leq n\leq N \\ \gcd(p,n)=1 \\ f(n)\neq 0}}F(f(pn)) ~+~\frac{2}{p}
\\
&\leq & 
\limsup_{N\to\infty}\frac{2}{N}\sum_{1\leq n\leq N\atop f(n)\neq 0} F(f(n)) ~+~\frac{2}{p}
\\
&\leq &
\frac{\epsilon}{4}+\frac{2}{p}
~\leq~ \frac{3\epsilon}{4}.
\end{eqnarray*}
This shows that $\overline{d}(\{n\in \N: f(n)\in J_2\setminus J_1\})\leq\epsilon$.
\end{proof}

\begin{proof}[Proof of \cref{lem:finding-K_1-and-K_2-for-D_4}]
Let $K\in\ca(f)$ and $\epsilon>0$ be arbitrary and define $J:=K\setminus\{0\}$. Since $K\in\ca(f)$, $\partial J$ is an $f$-null set ($f$-null sets were defined in \cref{def:n(f)}) and therefore, by \cref{lem:finding-K_1-and-K_2}, we can find sets $J_1\subset J_2\subset \C$ and $P\subset \P$ such that:
\begin{itemize}
\item[--] $f(p)\neq 0$ for all $p\in P$; 
\item[--] $\sum_{p\in P}\tfrac{1}{p}=\infty$;
\item[--] $f(p)J_1\subset J\subset f(p)J_2$ for all $p\in P$;
\item[--] $\overline{d}(\{n\in \N: f(n)\in J_2\setminus J_1\})\leq\epsilon $.
\end{itemize}  
Define
$$
K_1:=
\begin{cases}
J_1\cup\{0\},&\text{if $0\in K$}
\\
J_1,&\text{if $0\notin K$}
\end{cases}
\quad\text{and}\quad
K_2:=
\begin{cases}
J_2\cup\{0\},&\text{if $0\in K$}
\\
J_2,&\text{if $0\notin K$}.
\end{cases}
$$
It is now straightforward to check that $P$, $K_1$ and $K_2$ satisfy the conclusion of \cref{lem:finding-K_1-and-K_2-for-D_4}.
\end{proof}

We are now in position to give a proof of \cref{prop:finding-K_1-and-K_2}.

\begin{proof}[Proof of \cref{prop:finding-K_1-and-K_2}]
We start with the case $E\in\Dthree{\infty}$ and $\overline{d}(E)>0$.
Hence $E$ is of the form
$
E(\vec{f},K):=\{n\in\N:\vec{f}(n)\in K\},
$
where $K$ are arbitrary subsets of $\C^r$ and $\vec{f}:=(f_1,\ldots,f_r)$ is multiplicative function with at least one concentration point. Hence there exist $\vec{z}=(z_1,\ldots,z_r)\in\C^r$ and a set of primes $P\subset\P$ with $\sum_{p\in P}\frac{1}{p}=\infty$ and $f_i(p)=z_i$ for all $p\in P$ and all $1\leq i\leq r$.
Take
$$
E_1:=E_2:=\{n\in\N:\vec{f}\cdot \vec{z}\in K\},
$$
where $\vec{f}\cdot \vec{z}=(f_1(n)z_1,\ldots,f_r(n)z_r)\in\C^r$.
Note that $E_2\setminus E_1=\emptyset$ and therefore $\overline{d}(E_2\setminus E_1)=0$.
Also, for all $p\in P$ and $n\in\N$ with $\gcd(n,p)=1$, we have
$$
np\in E \quad\iff\quad \vec{f}(np)\in K\quad\iff\quad \vec{f}\cdot \vec{z}\in K\quad\iff\quad n\in E_1=E_2.
$$
This shows that $\1_{E_1}(n)= \1_{E}(n p)= \1_{E_2}(n)$ for all $p\in P$ and $n\in\N$ with $\gcd(n,p)=1$.

Next, we deal with the case $E\in\Dfour$. By the definition of $\Dfour$ there exist $f\in\cm$ with $\|f\|_1\neq 0$ and $K\in\ca(f)$ such that $E=E(f,K)=\{n\in\N: f(n)\in K\}$.
According to \cref{lem:finding-K_1-and-K_2-for-D_4}, we can find sets $K_1,K_2\subset \C$ and a set of prime numbers $P\subset \P$ satisfying:
\begin{enumerate}
[label=\text{(\arabic*)}~, ref=\text{(\arabic*)}, leftmargin=*]
\item\label{KOC-property-0} $f(p)\neq 0$ for all $p\in P$;
\item\label{KOC-property-1} $\sum_{p\in P}\tfrac{1}{p}=\infty$;
\item\label{KOC-property-2} $f(p) K_1\subset K \subset f(p)K_2$ for all $p\in P$;
\item\label{KOC-property-4} $\overline{d}(\{n\in \N: f(n)\in K_2\setminus K_1\})\leq\epsilon $.
\end{enumerate}  

Define $E_1:=\{n\in\N: f(n)\in K_1\}$ and $E_2:=\{n\in\N: f(n)\in K_2\}$. It follows from property \ref{KOC-property-4} that $\overline{d}(E_2\setminus E_1)\leq \vep$.
Using properties \ref{KOC-property-0} and \ref{KOC-property-2}, we deduce that $K_1\subset (f(p))^{-1}K\subset K_2$ for all $p\in P$. Also, if $p\in P$ and $\gcd(n,p)=1$, then
$$
np\in E\iff f(np)\in K \iff f(n)  \in (f(p))^{-1}K.
$$
It follows that $\1_{E_1}(n)\leq \1_{E}(n p)\leq \1_{E_2}(n)$ for all $p\in P$ and $n\in\N$ with $\gcd(n,p)=1$, which completes the proof.
\end{proof}

\subsection{Proof of \cref{prop:pKOc}}
\label{appendix:KOC}

Before embarking on the proof \cref{prop:pKOc} we formulate and prove the following variant of the classical Tur{\'a}n-Kubilius inequality.


\begin{Lemma}\label{lem:lB6}
Let $P$ be a finite subset of $\P$ and let
$w(n):=\sum_{p\in P}\1_{p\mid n}$, where $\1_{p\mid n}=1$ if $p\mid n$ and $\1_{p\mid n}=0$ otherwise, and $m:=\sum_{p\in P} \frac{1}{p}$.
Then,
\begin{equation}
\label{eqn:Turan-Kubilius-2}
\sum_{n\leq x}(w(n)-m)^2= \Oh(x m+|P|^2 ).
\end{equation}
\end{Lemma}

\begin{proof}
First we expand the left hand side of
\eqref{eqn:Turan-Kubilius-2} and get
\begin{eqnarray*}
\sum_{n\leq x}(w(n)-m)^2
&=& \Sigma'-
2\Sigma''+\Sigma''',
\end{eqnarray*}
where
$$
\Sigma':=\sum_{p,q\in P}\sum_{n\leq x} 
\1_{p\mid n}\1_{q\mid n},
\qquad
\Sigma'':=
m \sum_{p\in P}\sum_{n\leq x}\1_{p\mid n}
\quad\text{and}\quad
\Sigma''':=x m^2.
$$
Note that $\sum_{n\leq x}\1_{p\mid n}=\frac{x}{p}+ \Oh(1)$ and hence
$$
\Sigma''=m\left(\sum_{p\in P}\frac{x}{p}+\Oh(|P|)\right) =x m^2+\Oh(m|P|).
$$
Since $m|P|\leq |P|^2$, we get $\Sigma''=x m^2+\Oh(|P|^2).$

Next observe that
$\1_{p\mid n}\1_{q\mid n} = \1_{pq\mid n}$ unless $p = q$.
Therefore
\begin{equation}
\label{eqn:tk-1}
\Sigma'=
\sum_{p,q\in P}
\sum_{n\leq x} \1_{p q \mid n}+\sum_{p\in P}
\sum_{n\leq x} \big(\1_{p\mid n}-\1_{p^2\mid n}\big).
\end{equation}
We can estimate $\sum_{n\leq x} \1_{p q \mid n}=\frac{x}{pq}+ \Oh(1)$ and
$$
\sum_{p\in P}\sum_{n\leq x} \big(\1_{p\mid n}-\1_{p^2\mid n}\big)
~\leq~ 
\sum_{p\in P}\sum_{n\leq x} \1_{p\mid n}
~=~
\Oh\left(x m+|P|\right).
$$
Hence \eqref{eqn:tk-1} can be written as
$$
\Sigma'=\sum_{p,q\in P}\frac{x}{pq}+\Oh(x m+|P|^2)
=x m^2+\Oh(x m+|P|^2).
$$
Putting everything together we conclude that
\begin{eqnarray*}
\Sigma'-
2\Sigma''+\Sigma'''=\Oh(x m+|P|^2).
\end{eqnarray*}
\end{proof}

\begin{proof}[Proof of \cref{prop:pKOc}]
In what follows $y=y(x)$ will be a slowly growing function, the conditions for the rate of growth being clear from the context.
Instead of showing norm-convergence in \eqref{eqn:pKOc-4} we will show that
\begin{equation}\label{eqn:pKOc-4-new}
\sup_{u\in\Hilb\atop \|u\|\leq 1}\left|\sum_{n\leq x} \la \s(n)u, u_n\ra\right|= \oh(x)+\Oh(x\|\t_1-\t_2\|_1).
\end{equation}
Let $u\in\Hilb$ with $\|u\|\leq1$ be arbitrary.
We have
\begin{equation*}
\begin{split}
&\left|\sum_{n\leq x} \la \s(n)u, u_n\ra\right|
~=~
\frac{1}{m_y}
\left|\sum_{n\leq x}m_y\la \s(n)u,u_n\ra \right|
\\
\leq~
&
\frac{1}{m_y}
\left|\sum_{n\leq x} w_y(n)\la  \s(n)u,u_n\ra\right|+
\frac{1}{m_y}\left|\sum_{n\leq x} (m_y-w_y(n))\la \s(n)u,u_n\ra
\right|
\\
\leq~
&
\frac{1}{m_y}
\left|\sum_{n\leq x}w_y(n)\la \s(n)u,u_n\ra\right|
+
\frac1{m_y}\left(\sum_{n\leq x} (w_y(n)-m_y)^2 \right)^{1/2}
\left(\sum_{n\leq x} \la \s(n)u,u_n\ra^2 \right)^{1/2}.
\end{split}
\end{equation*}

We have used the Cauchy-Schwarz inequality in the last line.

Applying \cref{lem:lB6}, we get
\begin{eqnarray*}
\left|\sum_{n\leq x} \la\s(n)u,u_n\ra\right|
&\leq &
\frac{1}{m_y}\left|\sum_{n\leq x}w_y(n)\la \s(n)u,u_n\ra\right|
\\
&&\qquad+~\Oh\Big(\frac{(m_y x+|P_y|^2)^{1/2}x^{1/2}}{m_y}\Big).
\end{eqnarray*}
Let us assume that $y=y(x)$ is growing sufficiently slow so that
$$
\frac{(m_y x+|P_y|^2)^{1/2}x^{1/2}}{m_y}\leq
\Oh\Big(\frac{x}{\sqrt{m_y}}\Big).
$$
Hence
\begin{eqnarray*}
\left|\sum_{n\leq x} \la\s(n)u, u_n\ra\right|
&\leq &\frac{1}{m_y}\left|\sum_{n\leq x}\sum_{p\in P_y}
\1_{p\mid n}\la \s(n)u,u_n\ra\right|
+\Oh\Big(\frac{x}{\sqrt{m_y}}\Big)
\\
&\leq&\frac{1}{m_y}\left|\sum_{p\in P_y}\sum_{n\leq x/p}
\la \s(np)u,u_{np}\ra\right|
+\Oh\Big(\frac{x}{\sqrt{m_y}}\Big).
\end{eqnarray*}
Note that the cardinality of the set $\{n\leq x/p: \gcd(n, p)\neq 1\}$ does not exceed $x/p^2$. Since $\s$, $\t_1$, $\t_2$, $\rfctn$ and $u_n$ are bounded, it follows from \eqref{eqn:pKOc-2} that
\[
\left|\sum_{n\leq x/p}\la \s(np)u,u_{np}\ra~-~
\sum_{n\leq x/p}\la \t_1(n)\rfctn(p)u,u_{np}\ra\right|
=\Oh\Big(\frac{x}{p^2}\Big)+\Oh\Big(\frac{x}{p}\|\t_1-\t_2\|_1\Big).
\]
This implies that
\begin{eqnarray*}
\frac{1}{m_y}\left|\sum_{p\in P_y}\sum_{n\leq x/p}
\la \s(np)u,u_{np}\ra\right|
&=&
\frac{1}{m_y}\left|\sum_{p\in P_y}\sum_{n\leq x/p}
\la \t_1(n)\rfctn(p)u,u_{np}\ra\right|\\
&&~+\Oh\Big(\frac{x}{\sqrt{m_y}}\Big)+\Oh\big(x\|\t_1-\t_2\|_1\big).
\end{eqnarray*}

Next, we set $P_{k,y} = P_y\cap \{n\in\N: 2^k \leq n < 2^{k+1}\}$.
Hence
\begin{eqnarray*}
\sum_{p\in P_y}\sum_{n\leq x/p}
\la \t_1(n)\rfctn(p)u,u_{np}\ra
&=&\sum_{k=0}^{\log_2 y}\sum_{p\in P_{k,y}}
\sum_{n\leq x/p}\la \t_1(n)\rfctn(p)u,u_{np}\ra.
\end{eqnarray*}
Combining all of the above we get
\begin{equation}
\label{equation:vdc3}
\begin{split}
\left|\sum_{n\leq x} \la \s(n)u,u_n\ra\right|
\leq
\frac{1}{m_y}\sum_{k=0}^{\log_2 y}\left|\sum_{p\in P_{k,y}}
\sum_{n\leq x/p}\la \t_1(n)u,\rfctn(p)u_{np}\ra\right|&
\\
+\Oh\Big(\frac{x}{\sqrt{m_y}}\Big)+&\Oh\Big(x\|\t_1-\t_2\|_1\Big).
\end{split}
\end{equation}
Let $A_{k,y}$ be defined as
$$
A_{k,y}=\sum_{p\in P_{k,y}}
\sum_{n\leq x/p}\la \t_1(n)u,\rfctn(p)u_{np}\ra=
\sum_{n\leq x/2^k}\Big\la \t_1(n)u,
\sum_{p\in P_{k,y}}\1_{n\leq x/p} \rfctn(p)u_{np}\Big\ra.
$$
Fixing $k$ and applying the Cauchy-Schwarz inequality again, we get
\begin{eqnarray*}
|A_{k,y}|&\leq &
\sum_{n\leq x/2^k}|\t_1(n)|~\left\|
\sum_{p\in P_{k,y}}\1_{n\leq x/p}\rfctn(p)
u_{np}\right\|
\\
&\leq &
\left(\sum_{n\leq x/2^k}|\t_1(n)|^2\right)^{\frac12}~
\left(\sum_{n\leq x/2^k}\left\|
\sum_{p\in P_{k,y}}\1_{n\leq x/p}\rfctn(p)u_{np}\right\|^2
\right)^{\frac12}\\
&\leq &
\Oh\left(\frac{x^{\frac12}}{2^{\frac{k}{2}}}\right)
\left(\sum_{n\leq x/2^k}
\sum_{p,q\in P_{k,y}}
\1_{n\leq x/p}\1_{n\leq x/q}
\rfctn(p)\rfctn(q)\la u_{np},u_{nq}\ra
\right)^{\frac12}\\
&\leq &
\Oh\left(\frac{x^{\frac12}}{2^{\frac{k}{2}}}\right)
\left|
\sum_{p,q\in P_{k,y}\atop p\neq q}
\sum_{n\leq \min\{x/p,x/q\}}
\rfctn(p)\rfctn(q)\la u_{np},u_{nq}\ra
\right|^{\frac12}+\Oh\left(\frac{x|P_{k,y}|^{\frac12}}{2^k}\right).
\\
&\leq &
\Oh\left(\frac{x^{\frac12}}{2^{\frac{k}{2}}}\right)
\left(
\sum_{p,q\in P_{k,y}\atop p\neq q}
\left|\sum_{n\leq \min\{x/p,x/q\}}\la u_{np},u_{nq}\ra
\right|\right)^{\frac12}+\Oh\left(\frac{x|P_{k,y}|^{\frac12}}{2^k}\right).
\end{eqnarray*}
Using the prime number theorem to estimate $|P_{k,y}|^{\frac12}$
we deduce that
$$
\frac{1}{m_y}\sum_{k=0}^{\log_2 y} \frac{x|P_{k,y}|^{\frac12}}{2^k}=
\Oh\Big(\frac{x}{m_y}\Big).
$$
Combining this with equation \eqref{equation:vdc3} and using
$|P_{k,y}|\leq 2^k$
we get
\begin{equation}
\label{vdc4}
\begin{split}
\left|\sum_{n\leq x} \la \s(n)u,u_n\ra\right|\leq
\Oh\left(\frac{x^{\frac12}}{m_y}\right)\sum_{k=0}^{\log_2 y}
\left(\frac{1}{|P_{k,y}|}
\sum_{p,q\in P_{k,y}\atop p\neq q}
\left|\sum_{n\leq \min\{x/p,x/q\}}\la u_{np},u_{nq}\ra
\right|\right)^{\frac12}&
\\
+\Oh\Big(\frac{x}{\sqrt{m_y}}\Big)+\Oh\big(x\|\t_1-\t_2\|_1\big).&
\end{split}
\end{equation}

Finally, if $y=y(x)$ is growing sufficiently slowly then, from \eqref{eqn:pKOc-3}, we obtain that
$$
\left|\sum_{n\leq \min\{x/p,x/q\}}
\la u_{np},u_{nq}\ra\right|\leq \frac{x}{y\log_2^2 y}
$$
for every $p,q \in P_y$ with $p\ne q$.
Note that $|P_{k,y}|\leq y$ and hence
$$
\frac{1}{|P_{k,y}|}\sum_{p,q\in P_{k,y}\atop p\neq q}\left|\sum_{n\leq \min\{x/p,x/q\}}\la u_{np},u_{nq}\ra\right|\leq \frac{1}{|P_{k,y}|}\sum_{p,q\in P_{k,y}\atop p\neq q}\frac{x}{y\log_2^2 y} \leq \frac{x}{\log_2^2 y}.
$$
Thus the inequality \eqref{vdc4} becomes
\begin{eqnarray*}
\label{equation:vdc5}
\left| \sum_{n\leq x} \la \s(n)u,u_n\ra \right|
&\leq&
\Oh\left(\frac{x^{\frac12}}{m_y}\right)\sum_{k=0}^{\log_2 y}
\frac{x^{\frac12}}{\log_2 y}
+\Oh\Big(\frac{x}{\sqrt{m_y}}\Big)+\Oh\big(x\|\t_1-\t_2\|_1\big)\\
&=&
\Oh\Big(\frac{x}{\sqrt{m_y}}\Big)+\Oh\big(x\|\t_1-\t_2\|_1\big).
\end{eqnarray*}
Since all the estimates above do not depend on $u$
but only on $\|u\|$, it follows that
\begin{equation*}
\label{equation:vdc6}
\sup_{\|u\|\leq1}~\left|\sum_{n\leq x} \la \s(n)u,u_n\ra\right|
=
\Oh\Big(\frac{x}{\sqrt{m_y}}\Big)+\Oh\big(x\|\t_1-\t_2\|_1\big).
\end{equation*}
This completes the proof.
\end{proof}

\section{Applications to the theory of uniform distribution}
\label{subsec:ud}

Recall (cf.\ \cref{ftnt-2} and \cref{def:nu-u.d.}) that a sequence $(x_n)_{n\in\N}$ of real numbers is \define{uniformly distributed mod $1$} if
$$\lim_{N\rightarrow\infty}\frac{1}{N}\sum_{n=1}^N f(\{x_n\})=\int_0^1 f(x)\, dx,\qquad\forall f\in C([0,1)).$$

This section is dedicated to proving the following generalization of \cref{thm:uniform-distribution}.

\begin{Theorem}
\label{thm:uniform-distribution-2}
Let $E=\{n_1<n_2<\ldots\}$ be a set that belongs to either $\Dthree{\infty}$ or $\Dfour$.
Suppose $h\colon (0,\infty)\to\R$ belongs to a Hardy field, has polynomial growth and satisfies $|h(t)- r(t)|\succ \log^2(t)$ for all polynomials $r\in\Q[t]$.
If $d(E)$ exists and is positive then the sequence $\big(h(n_j)\big)_{j\in\N}$ is uniformly distributed mod $1$. \end{Theorem}

It follows immediately from Propositions \ref{prop:e-sets-1} and \ref{prop:e-sets-2} that \cref{thm:uniform-distribution} is a special case of \cref{thm:uniform-distribution-2}.

In the proof of \cref{thm:uniform-distribution-2} we will be using the following result of Boshernitzan. 

\begin{Theorem}[see {\cite[Theorem 1.3]{Boshernitzan94}}]\label{thm:Boshernitzan94-thm1.3}
Let $\Hardy$ be a Hardy field and assume $h\in\Hardy$ has polynomial growth (i.e.\ $|h(t)|\prec t^n$ for some $n\in\N$).
Then $(h(n))_{n\in\N}$ is uniformly distributed $\bmod~1$ if and only if for every polynomial $r\in\Q[t]$ one has $|h(t)- r(t)|\succ \log(t)$.
\end{Theorem}

We will also need the following lemma.

\begin{Lemma}\label{lem:hardy-hilfslemma}
Let $\Hardy$ be a Hardy field and assume $g\in\Hardy$ satisfies $|g(t)|\succ \log^2(t)$. Then, for all $p,q\in\N$ with $p\neq q$,
\begin{equation}
\label{eq:bosh-crit-katai-1}
|g(pt)-g(qt)|\succ\log(t).
\end{equation}
\end{Lemma}

\begin{proof}
It suffices to show that for all $c>1$ one has
\begin{equation}
\label{eq:bosh-crit-katai-2}
|g(ct)-g(t)|\succ\log(t),
\end{equation}
because \eqref{eq:bosh-crit-katai-1} follows quickly from \eqref{eq:bosh-crit-katai-2} by change of variables. Suppose there exists a constant $c>1$ such that \eqref{eq:bosh-crit-katai-2} is not satisfied. 
Remembering that $g(ct)-g(t)$ belongs to a Hardy field, this means that there exist $t_0\in(0,\infty)$ and $M>0$ such that
$$
|g(ct)-g(t)|\leq M\log(t),\qquad\forall t\in[t_0,\infty).
$$
Define $a:=|g(t_0)|$ and $b:=M\log(ct_0)$.
It follows that
\begin{eqnarray*}
|g(c^n t_0)|
&=& \left| g(t_0) + \sum_{j=1}^{n} \left(g(c^jt_0)-g(c^{j-1}t_0)\right)\right|
\\
&\leq& a + \sum_{j=1}^{n} |g(c^jt_0)-g(c^{j-1}t_0)|
\\
&\leq& a + M\sum_{j=1}^{n} \log(c^{j-1}t_0)
\\
&\leq& a+ bn^2.
\end{eqnarray*}
However, $|g(t)|\succ \log^2(t)$ and hence $|g(c^n t_0)|\succ \log^2(c^n t_0)\geq b'n^2$ for some constant $b'$. This is a contradiction.
\end{proof}

\begin{proof}[Proof of \cref{thm:uniform-distribution-2}]
Let $E=\{n_1<n_2<\ldots\}$ be a set that belongs to either $\Dthree{\infty}$ or $\Dfour$ and assume $d(E)$ exists and is positive. Let $\Hardy$ be a Hardy field, let $h\in\Hardy$ and suppose $h$ has polynomial growth and satisfies $|h(t)- r(t)|\succ \log^2(t)$ for all polynomials $r\in\Q[t]$. We want to show that the sequence $\big(h(n_j)\big)_{j\in\N}$ is uniformly distributed mod~$1$.

In light of Weyl's criterion it suffices to show that for all
$k\in\Z\setminus\{0\}$ the averages
$$
\frac{1}{N}\sum_{j=1}^{N}e(k h(n_j))
$$
converge to $0$ as $N\to\infty$.
Since $d(E)$ exits and is positive, this is equivalent to
\begin{equation}
\label{eq:wc-hardy-1}
\lim_{N\to\infty}\frac{1}{N}\sum_{n=1}^{N} \1_E(n) e(k h(n))=0,\qquad\forall k\in\Z\setminus\{0\}.
\end{equation}
In view of \cref{thm:Katai-multipliactive-fibers-2}, to prove \eqref{eq:wc-hardy-1} it suffices to show that
\begin{equation}
\label{eq:wc-hardy-2}
\lim_{N\to\infty}\frac{1}{N}\sum_{n=1}^{N} e(k(h(pn)-h(qn)))=0,
\end{equation}
for all primes $p\neq q$.

We claim that the sequence $(h(pn)-h(qn))_{n\in\N}$ is uniformly distributed mod $1$. Once we have verified this claim, \eqref{eq:wc-hardy-2} follows immediately, because $\int_0^1 e(kx)\, dx=0$.

Note that $h(pt)-h(qt)$ belongs itself to a Hardy field. According to \cref{thm:Boshernitzan94-thm1.3}, $(h(pn)-h(qn))_{n\in\N}$ is uniformly distributed mod $1$ if and only if for all $r\in\Q[t]$,
\begin{equation}
\label{eq:wc-hardy-3}
|h(pt)-h(qt)- r(t)|\succ \log(t).
\end{equation}
Let $r(t)=c_kt^k+\ldots+c_1t+c_0 \in\Q[t]$ be arbitrary. Note that the value of $c_0$ has no influence on \eqref{eq:wc-hardy-3} and we can assume that $c_0=0$. Define a new polynomial $s(t):=b_kt^k+\ldots+b_1t$, where $b_i:=\tfrac{c_i}{p^i-q^i}$, $1\leq i\leq k$. A simple calculation shows that $r(t)=s(pt)-s(qt)$. Define $g(t):=h(t)-s(t)$. Then \eqref{eq:wc-hardy-3} can be written as
\begin{equation}
\label{eq:wc-hardy-4}
|g(pt)-g(qt)|\succ \log(t).
\end{equation}
However, since $s(t)\in Q[t]$, we have that $|g(t)|=|h(t)-s(t)|\succ \log^2(t)$ by our assumption. Therefore \eqref{eq:wc-hardy-4} follows directly \cref{lem:hardy-hilfslemma}. This completes the proof.
\end{proof}

\section{Applications to Ergodic Theory and proofs of \cref{cor:M-sets-are-tot.erg.sequences} and \cref{thm:ergodic-sequence}}
\label{sec:single-rec}
\label{sec:ET}

We start by recalling the following well-known characterizations of ergodic and totally ergodic sequences (see Definitions \ref{def:ergodic-sequence} and \ref{def:tot-erg-sequence}).


\begin{Theorem}\label{thm_erg-seq-char}
Let $(n_j)_{j\in\N}$ be a sequence in $\N$.
\begin{enumerate}
[label=$\text{(\alph{enumi})}$, ref=$\text{(\alph{enumi})}$, leftmargin=*]
\item\label{itm:erg-seq-char}
The sequence $(n_j)_{j\in\N}$ is ergodic if and only if for all $\alpha\in\R\setminus\Z$,
$$
\lim_{N\to\infty}\frac{1}{N}\sum_{j=1}^N e(n_j\alpha) =0.
$$
\item\label{itm:tot-erg-seq-char}
The sequence $(n_j)_{j\in\N}$ is totally ergodic if and only if for all $\alpha\in\R\setminus\Q$,
$$
\lim_{N\to\infty}\frac{1}{N}\sum_{j=1}^N e(n_j\alpha) =0.
$$
\end{enumerate}
\end{Theorem}

(It is not hard to see that both parts of \cref{thm_erg-seq-char} follow immediately from the spectral theorem.)

\cref{thm_erg-seq-char} allows us to derive the following corollary from \cref{thm:Katai-multipliactive-fibers-2}.

\begin{Corollary}
\label{cor:M-sets-are-tot.erg.sequences-d3d4}
Let $E=\{n_1<n_2<\ldots\}$ be a set that belongs to either $\Dthree{\infty}$ or $\Dfour$ and suppose $d(E)$ exists and is positive. Then $(n_j)_{j\in\N}$ is a totally ergodic sequence.
\end{Corollary}

\begin{proof}
It follows from part \ref{itm:tot-erg-seq-char} of \cref{thm_erg-seq-char} that it suffices to show that
\begin{equation}
\label{eqn:tot-erg-proof-1}
\lim_{N\to\infty}\frac{1}{N}\sum_{j=1}^N e(n_j\alpha) =0
\end{equation}
for all irrational $\alpha$. Since $d(E)$ exists and is positive, 
equation \eqref{eqn:tot-erg-proof-1} is equivalent to
\begin{equation}
\label{eqn:tot-erg-proof-2}
\lim_{N\to\infty}\frac{1}{N}\sum_{n=1}^N \1_E(n) e(n\alpha)=0.
\end{equation}
However, \eqref{eqn:tot-erg-proof-2} follows from \cref{thm:Katai-multipliactive-fibers-2} because for any irrational $\alpha$ the sequence $e(n\alpha)$ satisfies \eqref{eq:KOC-a}.
\end{proof}

Note that in view of Propositions \ref{prop:e-sets-1} and \ref{prop:e-sets-2}, \cref{cor:M-sets-are-tot.erg.sequences} follows directly from \cref{cor:M-sets-are-tot.erg.sequences-d3d4}.
We also have the following generalization of \cref{thm:ergodic-sequence}.

\begin{Theorem}
\label{thm:ergodic-sequence-2}
Let $E=\{n_1<n_2<\ldots\}$ be a set that belongs to either $\Dthree{\infty}$ or $\Dfour$. Suppose $h\colon (0,\infty)\to\R$ belongs to a Hardy field $\Hardy$, has polynomial growth and satisfies either $\log^2 t\prec h(t)\prec t$ or $t^k\prec h(t)\prec t^{k+1}$ for some $k\in\N$.
If $d(E)$ exists and is positive then $\big(\lfloor h(n_j)\rfloor\big)_{j\in\N}$ is an ergodic sequence.
\end{Theorem}

\begin{proof}[Proof of \cref{thm:ergodic-sequence-2} (cf.\ {\cite[Lemma 5.12]{BK09}})]
In view of \cref{thm_erg-seq-char}, part \ref{itm:erg-seq-char}, it suffices to show that for every $\alpha\in \R\setminus\Z$ we have
$$
\lim_{N\to\infty}\frac{1}{N}\sum_{j=1}^N e\big(\lfloor h(n_j)\rfloor\alpha\big)=0.
$$
We have $\lfloor h(n)\rfloor=h(n)-\{h(n)\}$.
Therefore $e\big(\lfloor h(n)\rfloor\alpha\big)=g\big(\alpha h(n_j),h(n_j)\big)$, where $g\colon \R^2\to\C$ is the function $g(x,y)=e\big(x-\alpha\{y\}\big)$. Note that $g$ is $1$-periodic and hence can be viewed as a function from $\T^2$ to $\C$. It thus suffices to show that
\begin{equation}
\label{eqn:mos-2}
\lim_{N\to\infty}\frac{1}{N}\sum_{j=1}^N g\big(\alpha h(n_j),h(n_j)\big)=0.
\end{equation}

Let $H:=\overline{\{(\alpha t \bmod 1, t\bmod 1): t\in\R\}}$. Note that $H$ is a closed subgroup of $\T^2$ and one has $H=\T^2$ if $\alpha$ is irrational and $H\subsetneq \T^2$ if $\alpha$ is rational.

Let $\mu_H$ denote the (normalized) Haar measure on $H$.
We claim that  $\int g\, d \mu_H=0$.
If $H=\T^2$ then $\int g\, d \mu_H=\int \left(\int g(x,y)\, dx\right)\, dy=\int 0\, dy=0$. If $H\subsetneq \T^2$, then $\alpha$ must be rational and hence
$$
\overline{\{(\alpha t \bmod 1, t\bmod 1): t\in\R\}}=\{(\alpha t \bmod 1, t\bmod 1): t\in\R\}.
$$
Therefore,
\begin{equation}
\label{eqn:mos-1}
\lim_{T\to\infty}\frac{1}{T}\int_{0}^T f(\alpha t,t)\, dt =\int f\, d\mu_H, 
\end{equation}
for all continuous $f\colon H\to\C$. (Indeed, the left hand side of \eqref{eqn:mos-1} describes an invariant probability measure on $H$ and any invariant probability measure must coincide with $\mu_H$, by uniqueness of Haar measures.)
Thus, we have
\begin{eqnarray*}
\int g\, d\mu_H
&=& \lim_{T\to\infty}\frac{1}{T}\int_{0}^T e(\alpha t-\alpha\{t\})\, dt
\\
&=&
\lim_{T\to\infty}\frac{1}{T}\int_{0}^T e(\alpha \lfloor t\rfloor)\, dt~=~0.
\end{eqnarray*}
Since $\int g\, d\mu_H=0$ and $g$ is Riemann integrable, to show \eqref{eqn:mos-2} it suffices to show that the sequence $\big(\alpha h(n_j),h(n_j)\big)_{j\in\N}$ is uniformly distributed in $H$.
Since any group character of $H$ comes from a character on $\T^2$ and the non-trivial characters of $H$ are described by $\{(x,y)\mapsto e(\ell x+my):\ell,m\in\Z,~\alpha \ell+m\neq 0\}$, it follows from Weyl's equdistribution criterion that $\big(\alpha h(n_j),h(n_j)\big)_{j\in\N}$ is uniformly distributed in $H$ if and only if for all $(\ell,m)\in\Z^2$ that satisfy $\alpha \ell+m\neq 0$ one has
\begin{equation}\label{eq_proof_uniformdistribution}
\lim_{N\to\infty}\frac{1}{N}\sum_{j=1}^N e\big((\ell\alpha+m)h(n_j)\big)=0.
\end{equation}
Since $h\in\Hardy$ has polynomial growth and satisfies $n^{k-1}\prec h(t)\prec n^k$, we conclude that $(\ell\alpha+m)h(n)$ also belongs to $\Hardy$, has polynomial growth and satisfies $|(\ell\alpha+m)h(t)- r(t)|\succ \log^2(t)$ for all $r\in\Q[t]$.
It follows from \cref{thm:uniform-distribution-2} that the sequence $\big((\ell\alpha+m)h(n_j)\big)_{j\in\N}$ is uniformly distributed mod $1$. 
This implies that
$$
\lim_{N\to\infty}\frac{1}{N}\sum_{j=1}^N e\big((\ell\alpha+m)h(n_j)\big)=\int_{0}^1 e(x)\, dx=0
$$
and we conclude that \eqref{eq_proof_uniformdistribution} holds.
\end{proof}

\bibliographystyle{siam}

\providecommand{\noopsort}[1]{} 

\allowdisplaybreaks
\small
\bibliography{BibMF}

\bigskip
\footnotesize
\noindent
Vitaly Bergelson\\
\textsc{Department of Mathematics, Ohio State University, Columbus, OH 43210, USA}\par\nopagebreak
\noindent
\textit{E-mail address:}
\href{mailto:vitaly@math.ohio-state.edu}
{\texttt{vitaly@math.ohio-state.edu}}

\medskip

\noindent
J.\ Ku\l aga-Przymus\\
\textsc{Aix-Marseille Universit\'e, Centrale Marseille, CNRS, Institut de Math\'ematiques de Marseille, UMR7373, 39 Rue F.\ Joliot Curie 13453, Marseille, France}\\
\textsc{Faculty of Mathematics and Computer Science, Nicolaus Copernicus University, Chopina 12/18, 87-100 Toru\'{n}, Poland}\par\nopagebreak
\noindent
\textit{E-mail address:}
\href{mailto:joanna.kulaga@gmail.com}
{\texttt{joanna.kulaga@gmail.com}}

\medskip

\noindent
Mariusz Lema\'nczyk\\
\textsc{Faculty of Mathematics and Computer Science, Nicolaus Copernicus University, Chopina 12/18, 87-100 Toru\'{n}, Poland}\par\nopagebreak
\noindent
\textit{E-mail address:} 
\href{mailto:mlem@mat.umk.pl}
{\texttt{mlem@mat.umk.pl}}

\medskip

\noindent
Florian K.\ Richter\\
\textsc{Department of Mathematics, Ohio State University, Columbus, OH 43210, USA}\par\nopagebreak
\noindent
\textit{E-mail address:}
\href{mailto:richter.109@osu.edu}
{\texttt{richter.109@osu.edu}}

\end{document}